\documentclass[review,hidelinks,onefignum,onetabnum]{siamart251216}

\usepackage{lipsum}
\usepackage{amsfonts}
\usepackage{graphicx}
\usepackage{epstopdf}
\usepackage{algorithmic}

\usepackage{nicefrac}
\usepackage{longtable}
\usepackage{color}
\usepackage{graphicx, amssymb,graphics}

\ifpdf
\DeclareGraphicsExtensions{.eps,.pdf,.png,.jpg}
\else
\DeclareGraphicsExtensions{.eps}
\fi

\newtheorem{lem}{Lemma}

\newsiamremark{remark}{Remark}
\newsiamremark{hypothesis}{Hypothesis}
\crefname{hypothesis}{Hypothesis}{Hypotheses}
\newsiamthm{claim}{Claim}
\newsiamremark{fact}{Fact}
\crefname{fact}{Fact}{Facts}

\newcommand\dt {{\Delta t}}

\newcommand{\bv}{\boldsymbol{v}}
\newcommand{\bof}{\boldsymbol{f}}
\newcommand{\bu}{\boldsymbol{u}}
\newcommand{\be}{\boldsymbol{e}}
\newcommand{\bI}{\boldsymbol{I}}
\newcommand{\bpi}{\boldsymbol{\Pi}}
\newcommand{\bnu}{\boldsymbol{\nu}}
\newcommand{\bsigma}{\boldsymbol{\sigma}}

\headers{Entropy-stable method for the Dynamic van der Waals Model}{Y. Guo, Z. Qiao and X. Xu}

\title{Linear and entropy-stable scheme for the dynamic van der Waals model of non-isothermal liquid-vapor flows\thanks{Z. Qiao’s work is partially supported by NSFC/RGC Joint Research Scheme N\_PolyU5145/24 and the Hong Kong Polytechnic University RISE project 4-CDDQ. X. Xu's work is partially supported by NSFC projects (No. 12371415, 12461160275).}}

\author{
	Yunzhuo Guo \thanks{Department of Applied Mathematics, the Hong Kong Polytechnic University, Hung Hom, Hong Kong (yunzguo@polyu.edu.hk)}
	\and
	Zhonghua Qiao \thanks{Department of Applied Mathematics, the Hong Kong Polytechnic University, Hung Hom, Hong Kong (zhonghua.qiao@polyu.edu.hk)}
	\and
	Xianmin Xu \thanks{State Key Laboratory of Mathematical Sciences, Academy of Mathematics and Systems Science, Chinese Academy of Sciences, Beijing 100190, PR China (xmxu@lsec.cc.ac.cn)}
}

\usepackage{amsopn}

\begin{document}

\maketitle

\begin{abstract}
In this paper, we propose a linear and decoupled fully discrete scheme that preserves the entropy production property for the dynamic van der Waals model.
This model provides a thermodynamically consistent diffuse-interface framework for describing non-isothermal phase transitions in single-component liquid-vapor systems, containing the coupled conservation laws of mass, momentum and energy. 
We first introduce an equivalent reformulation of the governing equations that avoids the use of chain-rule identities, which generally fail under discrete differentiation. Further, a component-wise convex-splitting method is employed to calculate the discrete chemical potential and temperature. This treatment provides a discrete entropy-stable inequality without a time step restriction.
Based on the explicitly treated velocity, the scheme allows the three conservation equations to be solved linearly and sequentially in a decoupled manner. 
Since the momentum equation does not directly affect the entropy-stability analysis, we use a classical linear semi-implicit scheme to solve it.
Numerical experiments verify the expected numerical accuracy. Simulations of droplet evaporation and two- and three-dimensional Leidenfrost effects further demonstrate the proposed method can capture complex non-isothermal phase-transition dynamics while numerically reproducing the discrete entropy production behavior.
\end{abstract}



\section{Introduction} 
Non-isothermal processes are ubiquitous in nature, and the models used to describe them depend on the number of components and phases involved. The dynamic van der Waals model, also known as the dynamic van der Waals theory\cite{Onuki-PRE-05}, provides a thermodynamically consistent continuum framework for describing non-isothermal phase transitions in single-component liquid-vapor systems. These transitions exhibit fascinating phenomena, including the Leidenfrost effect \cite{Leidenfrost-13,Bouillant-2018} and thermocapillary effects \cite{JFM_2002,Guo-Lin-2015}. The ability to simulate such processes under dynamical thermal conditions makes the dynamic van der Waals model a powerful tool for understanding complex non-isothermal phase behavior.
This model is derived from a Helmholtz free energy, which is the equilibrium mean field theory of hard spherical particles subjected to long-range attractive forces \cite{van-der-Waals-79,vanDerWaals1988}.
\begin{equation}
	f(n,T)=n k_B T \left[\ln \left(\lambda_{\mathrm{th}}^3 n\right)-1-\ln \left(1-v_0 n\right)\right]-\varepsilon v_0 n^2  \nonumber
\end{equation}
where $n$ is the molecular number density, $T$ is the temperature, $v_0=a^3$ is the molecular volume, and $\varepsilon$ is the intermolecular attractive potential. $\lambda_{\mathrm{th}} = \hbar( 2 \pi / m_0 k_B T )^{1/ 2}$ is the thermal de Broglie wavelength, where $m_0$ and $\hbar$ are the molecular mass and Planck constant, respectively. 
From the free energy, the thermodynamic relations lead to explicit expressions for the internal energy density and entropy density:
\begin{align}
	& s(n,T) = -nk_B \left(3\ln \hbar + \frac{3}{2}\ln \frac{2\pi}{m_0 k_B T} + \ln \frac{n}{1-v_0 n}  \right)  + \frac{5 nk_B}{2}, \label{s(n,T)} \\
	& e(n,T) = \frac{3nk_B T}{2}-\varepsilon v_0 n^2. \label{e(n,T)}
\end{align}
A gradient contribution of density can also be incorporated into the free energy for inhomogeneous systems in thermodynamic equilibrium \cite{van-der-Waals-79}:
\begin{align}
	\hat{f}(n,T) = f(n,T) + \frac{M}{2} |\nabla n|^2. \nonumber
\end{align}

It is common to choose the number density $n$ and the temperature $T$ as the primary unknowns, with the internal energy $e$ and entropy $s$ regarded as functions of $n$ and $T$ \cite{Smejkal-2017,Kou-18,Kou-18b}. Based on \eqref{s(n,T)} and \eqref{e(n,T)}, an equivalent transformation can be applied to rewrite the independent variables as the number density $n$ and the internal energy $e$:
\begin{align}
	& s(n,e) = -nk_B \left( 3\ln \hbar + \frac{3}{2}\ln \frac{3\pi n}{m_0} - \frac{3}{2}\ln (e+\varepsilon v_0 n^2) + \ln \frac{n}{1-v_0n} \right) + \frac{5}{2} nk_B, \nonumber \\
	&  \frac{1}{T} = \frac{\partial s}{\partial e} =  \frac{3k_Bn}{2(e+\varepsilon v_0 n^2)}.\nonumber
\end{align}
The dynamic van der Waals model uses latter. When non-uniform temperature is taken into account, the dynamic van der Waals model introduces gradient contributions to the internal energy density and entropy density respectively:
\begin{align}
	\hat{e} = e + \frac{K}{2}|\nabla n|^2, \quad \hat{s} = s - \frac{C}{2}|\nabla n|^2. \nonumber
\end{align}
The governing equations are based on the conservation of mass, momentum, and energy:
\begin{align}
	& \frac{\partial n}{\partial t} + \nabla \cdot(n \bu )=0, \label{govern 1} \\
	& m_0n\left(\frac{ \partial \boldsymbol{u}}{\partial t}+\bu \cdot \nabla \bu \right)=\nabla \cdot (\boldsymbol{\sigma}-\bpi)-m_0 n g \be_z,   \label{govern 2}\\
	& \frac{\partial e_{tot}}{\partial t}+\nabla \cdot\left(e_{tot} \boldsymbol{u}\right)=\nabla \cdot\left( (\boldsymbol{\sigma}-\bpi) \boldsymbol{u} \right) + \nabla \cdot (\lambda\nabla T)-m_0 n g \bu\cdot\be_z, \label{govern 3} \\
	& T = \frac{2(e+\varepsilon v_0 n^2)}{3k_Bn},\label{govern 4}
\end{align}
where $e_{tot}=e+\frac{K}{2}|\nabla n|^2+\frac{m_0 n}{2}|\bu|^2$ is the total energy density, $\be_z$ is the upward unit vector in the vertical direction ($\be_y$ in 2D), $g$ is the gravitational acceleration, and $\lambda$ is the thermal conductivity. 
$\bsigma$ denotes the dissipative stress and $\bpi$ is the reversible stress tensor containing the pressure and density-gradient contributions. The latter can be derived from the second law of thermodynamics or the Onsager variational principle \cite{Onsager-PRE-1,Onsager-PRE-2}. The definitions are given below: 
\begin{align}
	& \boldsymbol{\sigma}=\eta\left(\nabla \boldsymbol{u}+\nabla \boldsymbol{u}^T\right)+\left(\xi-\frac{2 \eta}{3}\right) \nabla \cdot \boldsymbol{u} \boldsymbol{I}, \quad \text{$\eta$ and $\xi$ are viscosity coefficients,} \nonumber \\
	& \boldsymbol{\Pi}=\hat{p} \boldsymbol{I}+(K+CT) \nabla n \otimes \nabla n, \nonumber \\
    & \hat{p}=n \hat{\mu}-\hat{e}+T \hat{s}, \quad  \hat{\mu}=-T\left(\frac{\partial s}{\partial n}+\nabla \cdot \left(\frac{K+CT}{T} \nabla n\right)\right). \nonumber
\end{align}
Based on the mass and momentum equations, the energy balance \eqref{govern 3} could be also rewritten as 
\begin{equation}
	\frac{\partial \hat{e}}{\partial t}+\nabla \cdot\left(\hat{e} \boldsymbol{u}\right)= (\boldsymbol{\sigma}-\bpi):\nabla \boldsymbol{u}  + \nabla \cdot (\lambda\nabla T), \label{another energy equation}
\end{equation}
$\hat{e}$ is the internal energy with gradient contribution $\hat{e}:=e+\frac{K}{2}|\nabla n|^2$.

The boundary conditions can be different according to the specific modeling framework. Considering the complexity involved in both theoretical analysis and numerical computation, simple boundary conditions are often imposed. For example, one may use the homogeneous Neumann condition for the number density, $\bnu\cdot\nabla n=0$, a homogeneous Neumann or Dirichlet condition for the temperature, $\bnu\cdot\nabla T =0$ or $T|_\Gamma=T_0$, and the homogeneous Dirichlet condition for the velocity, $\bu|_\Gamma=\boldsymbol{0}$. The energy and entropy on the boundary can derive complex boundary conditions. In \cite{Onuki-PRE-05,Onuki-PRE-07}, the authors applied equilibrium boundary conditions to the governing equations, which assumes that the evolution near the boundary is considerably faster than that in the bulk region. In \cite{Qiao_Xu_Yang_Zhang_2026}, the authors extended the dynamic van der Waals model via the Onsager principle, thereby obtaining dynamic boundary conditions. 
Based on Waldmann’s approach \cite{Waldmann-1967}, Liu et al.\ developed suitable boundary conditions \cite{Xu-Liu-Qian-12}. These formulations were subsequently used in the simulations of droplet motion driven by temperature gradients \cite{Xu-Qian-PRE-12}, boiling phenomena \cite{Xu-Qian-PRE-13}, and the Leidenfrost effect \cite{Xu-Qian-PRE-14}, confirming the validity of the model even though energy conservation and entropy production were not enforced at the discrete level.

Energy conservation and entropy production are fundamental thermodynamic structures of the model. Under thermodynamically consistent boundary conditions, the boundary energy and entropy fluxes must also be taken into account. Preserving these physical structures is a major challenge in designing numerical schemes. 
Convex-splitting and implicit-explicit methods are widely used to construct efficient and stable numerical schemes for phase-field models \cite{chen16,CHEN-2019,chen22a,Guo24,LiQuanTang2022,ZhangWangTeng2024}. The ideas are also applicable to entropy-stable discretizations. Related work has been carried out by Kou et al.\ \cite{Kou-18,Kou-18b}, where an entropy-stable semi-discrete scheme was developed via a component-wise convex-splitting approach. Energy conservation was achieved by an artificial viscosity term, however, this also makes it impossible to generalize the numerical scheme to the fully discrete case.
Using a special jump operator, Landis et al.\ \cite{Liu-Landis-2015} developed an entropy-stable numerical scheme for the Navier-Stokes-Korteweg equations, which also provide an important framework for non-isothermal liquid-vapor systems. A number of impressive numerical results were published as well.

Although the incompressibility assumption is not suitable for single-component two-phase flows, it is reasonably valid in certain cases. The divergence-free velocity significantly reduces the computational complexity. Numerical simulations of Marangoni effects were reported in earlier years \cite{Sun-Liu-Xu-2009,Guo-Lin-Wang-2014}. Wang et al.\ developed fully discrete thermodynamically consistent numerical methods for incompressible multi-component models, including non-isothermal crystal growth models \cite{Li-Zhao-Wang-19} and non-isothermal hydrodynamic models of binary fluid flows \cite{Sun-Li-Zhao-Wang-20,Sun-Wang-24}. In addition, in \cite{Kou-26a}, the authors developed a non-isothermal framework for incompressible two-phase flow in porous media by coupling heat transfer with the two-phase flow equations.

In this work, we first introduce a suitable non-dimensionalization to reduce the difference of the physical scales and the number of independent parameters. 
The chain rule for variational derivatives is not valid under discrete temporal and spatial differentiation, which is an essential issue for fully discrete scheme. To address this issue, we equivalently rewrite the original equations using stress identities.  
After that, a linear, entropy-producing, fully discrete numerical scheme is proposed. The main idea comes from the observation that the mass balance equation \eqref{govern 1} is naturally decoupled by treating the velocity explicitly. The velocity in the energy equation is also treated explicitly, so that all three equations can be solved in a decoupled manner.
The chemical potential and temperature are defined through the component-wise convex-splitting approach, in which one variable is fixed while the splitting is performed with respect to the other, rather than decomposing the Hessian matrix directly. This strategy is commonly adopted in multi-component phase-field models \cite{Xiaolin-2017,kou-2026}. 
The momentum equation does not directly affect the entropy-production analysis, so we use a classical linear semi-implicit scheme to solve it. Since the number density is computed from the mass equation, the convex-splitting method also introduces no nonlinearity. In addition, a standard cell-centered finite difference method on staggered grids is employed for the spatial discretization. 

The rest of this paper is organized as follows. 
In view of the extreme scales of the physical quantities, the equations are first non-dimensionalized in Section \ref{sec:dimensionless}, along with the non-dimensional entropy analysis. In Section \ref{sec:numerical scheme}, the dimensionless system is equivalently rewritten. After that, a linear, decoupled, fully discrete scheme is constructed via the component-wise convex-splitting approach. The entropy-production property is proved based on the convexity. Section \ref{sec: simulations} is devoted to the verification of convergence orders and numerical simulations of droplet evaporation and the two- and three-dimensional Leidenfrost effect. A summary of the paper is given in Section \ref{sec: conclusion}.

\section{The dimensionless equations} \label{sec:dimensionless}
In order to improve numerical stability, reduce the difference among the physical scales, and simultaneously decrease the number of parameters, the governing equations are first non-dimensionalized. Before doing so, we adopt several assumptions commonly used in the literature \cite{Onuki-PRE-07,Xu-Qian-PRE-13,Xu-Qian-PRE-14}: (i) the shear and bulk viscosities are equal and proportional to the number density, $\eta(n) = \xi(n) = \nu_0 m_0 n$, where $\nu_0 = \eta / \rho$ is the kinematic viscosity independent of $n$; (ii) the thermal conductivity is also proportional to the number density, $\lambda = k_B \nu_0 n$; (iii) the gradient coefficient in energy is zero $K=0$ and $C$ is a positive constant, namely
\begin{equation*}
	\hat{e} =e, \quad e_{tot} = e+\frac{m_0n}{2}|\bu|^2,\quad \hat{s} = s - \frac{C}{2} |\nabla n|^2.
\end{equation*}
The viscosity and thermal conductivity coefficients can also be chosen to be either constants or nonlinear functions of the density and temperature. If the latter is adopted, the non-dimensionalization will be different.

The number density is naturally scaled by $\tilde{n}=v_0 n$, where $v_0$ is the molecular volume. It enforces that the dimensionless number density between $0$ and $1$. Here the tilde denotes dimensionless variables. Denote the space and time scales by $l_0$ and $t_0$, respectively, which further induce the velocity scale $\frac{l_0}{t_0}$:
\begin{align}
	& (\tilde{x},\tilde{y}) = (x,y)/l_0, \quad l_0=\left(\frac{C}{2k_Bv_0}\right)^{1/2}, \nonumber\\
	& \tilde{t}=t/t_0, \quad t_0 = \frac{l_0^2}{\nu_0}, \quad \tilde{\bu}=\bu \frac{t_0}{l_0}.\nonumber
\end{align}
The length reference $l_0$ guarantees that the gradient contribution and the entropy are scaled consistently in the non-dimensionalized equations. The time scale $t_0$ is the viscous relaxation time corresponding to the length scale $l_0$ and the Reynolds number is set to unity consequently. The entropy and internal energy are scaled in a natural way
\begin{align}
	\tilde{e} = \frac{v_0}{\varepsilon} e, \quad \tilde{s} = \frac{v_0}{k_B} s,\nonumber
\end{align}
while the temperature and chemical potential are treated by variational identity
\begin{align}
	& \frac{1}{T} = \frac{\partial s}{\partial e} = \frac{\partial \frac{k_B}{v_0} \tilde{s} }{\partial \frac{\varepsilon}{v_0} \tilde{e}} = \frac{k_B}{\varepsilon} \frac{1}{\tilde{T}}, \quad \tilde{T} = \frac{k_B}{\varepsilon} T,\nonumber \\
	& \hat{\mu}=-T\left(\frac{\partial s}{\partial n}+C\Delta n\right) = -\frac{\varepsilon}{k_B}\tilde{T} \cdot k_B\left(\frac{\partial \tilde{s}}{\partial \tilde{n}}+2\Delta \tilde{n}\right) = \varepsilon \tilde{\hat{\mu}}, \quad \tilde{\hat{\mu}} = \frac{1}{\varepsilon}\hat{\mu}.\nonumber
\end{align}

Consequently, the dimensionless equations are obtained after dropping the tildes:
\begin{align}
	& \frac{\partial n}{\partial t} + \nabla\cdot (n\bu) = 0, \label{dimensionless eq 1} \\
	& R\left(n\frac{\partial \bu}{\partial t} + n\bu\cdot\nabla\bu\right)=\nabla \cdot \left( R\boldsymbol{\sigma}-\bpi \right)-G n \be_z,\label{dimensionless eq 2} \\
	& \frac{\partial e}{\partial t} + \nabla\cdot ( e \bu) = (R\boldsymbol{\sigma}-\bpi):\nabla \bu + \nabla\cdot(n\nabla T), \label{dimensionless eq 3}\\
	& \frac{1}{T} = \frac{\partial s}{\partial e} = \frac{3 n}{2(e+n^2)}.\label{dimensionless eq 4}
\end{align}
with boundary conditions $\bu |_\Gamma = \boldsymbol{0}$, $\bnu\cdot\nabla n = 0$, and $\bnu\cdot\nabla T =0$ or $T|_\Gamma=T_0$ in simulations. The parameter $R=\frac{m_0 l_0^2}{t_0^2\varepsilon}$ is a dimensionless quantity related to the molecular kinetic energy and the intermolecular attractive potential. The parameter $G=\frac{m_0gl_0}{\varepsilon}$ represents gravity. The expressions for the dimensionless entropy, energy, and stress are given as follows:
\begin{align}
	& s(n,e)=-n\left(\frac{3}{2}\ln n+\ln\frac{n}{1-n}+M\right)+\frac{3}{2}n\ln(e+n^2),\quad M=\frac{3}{2}\ln \frac{3\pi\hbar^2}{m_0 \varepsilon v_0^{\frac{2}{3}}} - \frac{5}{2}, \nonumber \\
	&\boldsymbol{\sigma}=n\left(\nabla \boldsymbol{u}+\nabla \boldsymbol{u}^T\right)+\frac{1}{3}n\nabla \cdot \bu \boldsymbol{I}, \nonumber \\
    & E_{tot}=\left(e,1\right)+\frac{R}{2}\left( n,|\bu|^2\right), \quad \hat{s}=s-|\nabla n|^2,\quad \hat{S}=(\hat{s},1), \nonumber\\
	& \bpi=\hat{p} \boldsymbol{I}+2T \nabla n \otimes \nabla n, \quad \hat{p}=n \hat{\mu}-e+T \hat{s}, \quad \hat{\mu}=-T\left(\frac{\partial s}{\partial n}+2\Delta n\right) . \nonumber
\end{align}
The dimensionless equations \eqref{dimensionless eq 1}-\eqref{dimensionless eq 4} satisfy entropy production and total energy conservation, which is essential in numerical simulation.
\begin{theorem}\label{non-dimensional entropy and energy}
	The dimensionless equations \eqref{dimensionless eq 1}-\eqref{dimensionless eq 4} satisfy the following energy conservation and entropy production properties:
	\begin{itemize}
		\item[(i)] energy conservation
		\begin{equation}
			\frac{\partial}{\partial t} E_{tot} = \left(n, \bnu\cdot\nabla T\right)_\Gamma - \left(G n \be_z,\bu\right), \label{non-dimensional energy}
		\end{equation}
		\item[(ii)] entropy production
		\begin{align}
			\frac{\partial}{\partial t} \hat{S} &= \left( \frac{n}{T^2}\nabla T, \nabla T \right) + \left(\frac{n}{T}, \frac{R}{2}\left(\nabla \boldsymbol{u}+\nabla \boldsymbol{u}^T\right):\left(\nabla \boldsymbol{u}+\nabla \boldsymbol{u}^T\right) \right)  \label{non-dimensional entropy} \\
			& +  \left( \frac{Rn}{3T}, (\nabla\cdot\bu)^2\right)+\left( \frac{n}{T},\bnu\cdot\nabla T \right)_\Gamma \geq \left( \frac{n}{T},\bnu\cdot\nabla T \right)_\Gamma. \nonumber
		\end{align}
	\end{itemize}
	Here, the surface integral represents the heat flux, which vanishes under the homogeneous Neumann boundary condition for temperature. The effect of gravity is also included in the energy balance.
\end{theorem}
\begin{proof}
	First, we derive the energy conservation law. For the internal energy, the integration-by-parts formula yields
	\begin{equation}
		\frac{\partial}{\partial t} \left( e,1 \right) = \left(  (R\boldsymbol{\sigma}-\bpi):\nabla \bu, 1\right) + \left(n, \bnu\cdot\nabla T\right)_\Gamma. \label{nodim-energy-1}
	\end{equation}
	For the kinetic energy, we have
	\begin{align}
		&\frac{\partial}{\partial t} \frac{R}{2}\left( n,|\bu|^2\right)  = \frac{R}{2}\left( \frac{\partial n}{\partial t}, |\bu|^2 \right) + R\left(n\frac{\partial \bu}{\partial t},\bu \right)  \label{nodim-energy-2}\\
        & \quad\quad = \frac{R}{2}\left( n\bu, \nabla |\bu|^2 \right) + R\left(n\frac{\partial \bu}{\partial t},\bu \right) \nonumber\\
		& \quad\quad = R\left( n\bu, \bu \cdot \nabla \bu \right) - R\left( n\bu, \bu \cdot \nabla \bu \right) + \left( \nabla \cdot \left( R\boldsymbol{\sigma}-\bpi \right),\bu \right) - \left(G n \be_z,\bu\right)\nonumber \\
		& \quad\quad = -\left(  (R\boldsymbol{\sigma}-\bpi):\nabla \bu, 1\right)- \left(G n \be_z,\bu\right), \nonumber
	\end{align}
	where the mass equation \eqref{dimensionless eq 1} and the integration-by-parts formula have been used. \eqref{nodim-energy-1} and \eqref{nodim-energy-2} lead to the energy conservation property \eqref{non-dimensional energy}:
	\begin{equation}
		\frac{\partial}{\partial t} E_{tot} = \frac{\partial}{\partial t} \left( e,1 \right) + \frac{\partial}{\partial t} \frac{R}{2}\left( n,|\bu|^2\right) = \left(n, \bnu\cdot\nabla T\right)_\Gamma - \left(G n \be_z,\bu\right).\nonumber
	\end{equation}
	Now we consider the entropy part. Taking the time derivative of the entropy gives 
	\begin{align}
		\frac{\partial}{\partial t} \hat{S} = \left( \frac{\partial s}{\partial t},1 \right) + 2\left( \frac{\partial n}{\partial t}, \Delta n \right) = \left( \frac{\partial s}{\partial n},\frac{\partial n}{\partial t} \right) + \left( \frac{\partial s}{\partial e},\frac{\partial e}{\partial t} \right) + 2\left( \frac{\partial n}{\partial t}, \Delta n \right).\nonumber
	\end{align}
	Based on the definitions of the chemical potential and temperature, together with the mass and energy equations, we derive the following identities:
	\begin{align}
		&\frac{\partial}{\partial t} \hat{S}  = \left( \frac{\partial e}{\partial t}, \frac{1}{T} \right) - \left( \frac{\partial n}{\partial t}, \frac{\hat{\mu}}{T} \right) = \left( \frac{\partial e}{\partial t}, \frac{1}{T} \right) + \left( \nabla\cdot(n\bu), \frac{\hat{\mu}}{T} \right) \label{nodim-energy-4} \\
		& = \left( \frac{1}{T}, \nabla\cdot(n\nabla T) \right) + \left( \frac{1}{T}, (R\boldsymbol{\sigma}-\bpi):\nabla \bu \right) + \left( e\bu, \nabla\frac{1}{T} \right) - \left( n\bu, \nabla\frac{\hat{\mu}}{T} \right) \nonumber \\
		& = \left( \frac{n}{T^2}\nabla T, \nabla T \right) + \left(\frac{n}{T}, \frac{R}{2}\left(\nabla \boldsymbol{u}+\nabla \boldsymbol{u}^T\right):\left(\nabla \boldsymbol{u}+\nabla \boldsymbol{u}^T\right) \right)+\left( \frac{Rn}{3T}, (\nabla\cdot\bu)^2\right)  \nonumber \\
		& \qquad + \left( \frac{n}{T},\bnu\cdot\nabla T \right)_\Gamma+\left( \nabla \cdot \frac{\bpi}{T} + e\nabla\frac{1}{T} - n \nabla \frac{\hat{\mu}}{T}, \bu \right).\nonumber
	\end{align} 
	It is clear that the first three terms on the right-hand side are nonnegative. The last term vanishes by a direct calculation
	\begin{align}
		\frac{\bpi}{T} & = \left(n\frac{\hat{\mu}}{T}-\frac{e}{T}+s-|\nabla n|^2\right)\bI+2\nabla n \otimes \nabla n, \label{nodim-energy-5}\\
		\nabla\cdot\frac{\bpi}{T}& = \nabla\left(n\frac{\hat{\mu}}{T}-\frac{e}{T}+s-|\nabla n|^2\right) + 2\nabla\cdot(\nabla n \otimes \nabla n) \label{nodim-energy-6} \\
		& = \nabla\left(n\frac{\hat{\mu}}{T}-\frac{e}{T}\right)+\frac{\partial s}{\partial n}\nabla n+\frac{\partial s}{\partial e}\nabla e + 2(\Delta n)\nabla n \nonumber \\
		& = \nabla\left(n\frac{\hat{\mu}}{T}-\frac{e}{T}\right) - \frac{\hat{\mu}}{T}\nabla n + \frac{1}{T}\nabla e \nonumber \\
		& = n\nabla \frac{\hat{\mu}}{T} - e\nabla \frac{1}{T}, \nonumber
	\end{align}
	where the following chain-rule identities are essential:
	\begin{align}
		\nabla s & = \frac{\partial s}{\partial n}\nabla n+\frac{\partial s}{\partial e}\nabla e,  \label{nodim-energy-7} \\
		2\nabla\cdot(\nabla n \otimes \nabla n) & = \nabla |\nabla n|^2 + 2(\Delta n)\nabla n. \label{nodim-energy-8}
	\end{align}
	\eqref{nodim-energy-6} and \eqref{nodim-energy-4} lead to the entropy production law \eqref{non-dimensional entropy}
	\begin{align}
		\frac{\partial}{\partial t} \hat{S} = \left( \frac{n}{T^2}\nabla T, \nabla T \right)& + \left(\frac{n}{T}, \frac{R}{2}\left(\nabla \boldsymbol{u}+\nabla \boldsymbol{u}^T\right):\left(\nabla \boldsymbol{u}+\nabla \boldsymbol{u}^T\right) \right) + \left( \frac{Rn}{3T}, (\nabla\cdot\bu)^2\right) \nonumber \\
		&  +\left( \frac{n}{T},\bnu\cdot\nabla T \right)_\Gamma \geq \left( \frac{n}{T},\bnu\cdot\nabla T \right)_\Gamma. \nonumber
	\end{align}
\end{proof}

\section{The fully discrete numerical scheme} \label{sec:numerical scheme}
In this section, we propose a fully discrete numerical scheme and establish the entropy-stability analysis. For simplicity, we present the two-dimensional case only, since the three-dimensional case can be extended in a straightforward manner. Before doing so, it is necessary to reformulate the dimensionless system into an equivalent form.

\subsection{The equivalent governing equations}
The dimensionless system \eqref{dimensionless eq 1}-\eqref{dimensionless eq 4} is not yet in a form suitable for numerical discretization. In fact, because the chain rule fails to hold under both temporal and spatial discretizations, we are unable to mimic the procedure of the Theorem \ref{non-dimensional entropy and energy} to prove the entropy increase property in the discrete case. Observing that the stress $\bpi$ appears in divergence form in the momentum equation and can be expressed in divergence form in the energy equation, we derive an equivalent system based on \eqref{nodim-energy-6} to avoid the impact of the chain rule. The divergence expression of the stress $\bpi$ constitutes the core of the equivalent transformation and can be obtained through careful computation:
\begin{align}
    \nabla \cdot \bpi & = \nabla \cdot\left( \frac{\bpi}{T} \cdot T \right) = T \nabla \cdot \frac{\bpi}{T} + \frac{\bpi}{T} \nabla T \label{div pi 2} \\
    & = T\left( n\nabla \frac{\hat{\mu}}{T} - e\nabla \frac{1}{T} \right) + \frac{1}{T}\left(\hat{p} \boldsymbol{I}+2T \nabla n \otimes \nabla n\right) \nabla T \nonumber  \\
    & = T\left( n\nabla \frac{\hat{\mu}}{T} - e\nabla \frac{1}{T} \right) + \frac{n \hat{\mu}-e}{T}\nabla T+\left( \hat{s}\bI + 2\nabla n \otimes \nabla n \right)\nabla T     \nonumber  \\
    & = n\nabla \hat{\mu} -\frac{n\hat{\mu}}{T}\nabla T + \frac{e}{T}\nabla T + \frac{n \hat{\mu}-e}{T}\nabla T+\left( \hat{s}\bI + 2\nabla n \otimes \nabla n \right)\nabla T  \nonumber \\
    & = n\nabla \hat{\mu} + \left( \hat{s}\bI + 2\nabla n \otimes \nabla n \right) \nabla T \nonumber
\end{align}
where identity \eqref{nodim-energy-6} and the definition of $\bpi$ and $\hat{p}$ are applied.

For the momentum equation, a direct substitution of identity \eqref{div pi 2} leads to the reformulation:
\begin{align}
	R\left(n\frac{\partial \bu}{\partial t} + n\bu\cdot\nabla\bu\right)&=\nabla \cdot \left( R\boldsymbol{\sigma}-\bpi \right) - Gn\be_z  \label{equivalent momentum eq}\\
	& = R\nabla \cdot \boldsymbol{\sigma} - \left( n\nabla \hat{\mu} + \left( \hat{s}\bI + 2\nabla n \otimes \nabla n \right)\nabla T \right) - Gn\be_z.\nonumber
\end{align}
As for the energy equation, the stress $\bpi$ should be rewritten into divergence form first:
\begin{align}
	\frac{\partial e}{\partial t} + \nabla\cdot ( e \bu) &= (R\boldsymbol{\sigma}-\bpi):\nabla \bu + \nabla\cdot(n\nabla T)  \label{equivalent energy eq 1}  \\
	&  = R\boldsymbol{\sigma}:\nabla \bu - \nabla\cdot(\bpi\bu) + \bu \cdot (\nabla\cdot\bpi) + \nabla\cdot(n\nabla T). \nonumber
\end{align}
The definition of $\bpi$ and \eqref{div pi 2} show that
\begin{align}
	 - \nabla\cdot(\bpi\bu) &= - \nabla\cdot\left( (n \hat{\mu}-e+T \hat{s})\bu + 2T \nabla n \otimes \nabla n \cdot \bu \right)\label{equivalent energy eq 2}   \\
	&  = \nabla\cdot ( e \bu) -n\bu \cdot \nabla \hat{\mu} -\hat{\mu}\nabla\cdot(n\bu) - \hat{s}\bu \cdot \nabla T \nonumber \\
    & \quad\quad -T\nabla\cdot(\hat{s}\bu) - 2\nabla\cdot\left( T \nabla n \otimes \nabla n \cdot \bu \right),  \nonumber
\end{align}
\begin{align}
	& - 2\nabla\cdot\left( T \nabla n \otimes \nabla n \cdot \bu \right) = -2\nabla\cdot\left( (\bu \cdot \nabla n)T\nabla n \right) \label{equivalent energy eq 3}\\
    & \quad\quad\quad= -2(\bu \cdot \nabla n)\nabla\cdot(T\nabla n)-2T\nabla(\bu \cdot \nabla n)\cdot \nabla n \nonumber \\
	& \quad\quad\quad = -2(\bu \cdot \nabla n)(\nabla T \cdot\nabla n)-2(\bu \cdot \nabla n)T \Delta n-2T\nabla(\bu \cdot \nabla n)\cdot \nabla n \nonumber\\
	& \quad\quad\quad = -2(\bu \cdot \nabla n)(\nabla T \cdot\nabla n) - 2T\nabla \cdot (\nabla n \otimes \nabla n \cdot \bu), \nonumber
\end{align}
and 
\begin{align}
	\bu \cdot (\nabla\cdot\bpi) & = \bu \cdot (n\nabla \hat{\mu} + \left( \hat{s}\bI + 2\nabla n \otimes \nabla n \right)\nabla T) \label{equivalent energy eq 4}   \\
	& = n\bu \cdot \nabla \hat{\mu} + \hat{s}\bu \cdot \nabla T + 2 (\nabla T\cdot\nabla n)(\bu \cdot\nabla n).\nonumber
\end{align}
Equations \eqref{equivalent energy eq 1}-\eqref{equivalent energy eq 4} present the energy equation and, together with mass balance and \eqref{equivalent momentum eq}, yield the equivalent governing equations
\begin{align}
	& \frac{\partial n}{\partial t} + \nabla\cdot(n\bu) =0, \label{equiva 1} \\
	&R\left(n\frac{\partial \bu}{\partial t} + n\bu\cdot\nabla\bu\right)=R\nabla \cdot \boldsymbol{\sigma} - \left( n\nabla \hat{\mu} + \left( \hat{s}\bI + 2\nabla n \otimes \nabla n \right)\nabla T \right)- Gn\be_z, \label{equiva 2} \\
	& \frac{\partial e}{\partial t} + \hat{\mu}\nabla\cdot(n\bu) + T\nabla\cdot(\hat{s}\bu)  = R\boldsymbol{\sigma}:\nabla \bu + \nabla\cdot(n\nabla T) - 2T\nabla \cdot (\nabla n \otimes \nabla n \cdot \bu).\label{equiva 3}
\end{align}
The identities \eqref{nodim-energy-6} and \eqref{div pi 2} are already embedded implicitly in the equations \eqref{equiva 1}-\eqref{equiva 3}. The entropy increase can be demonstrated without recourse to the chain rule \eqref{nodim-energy-8}, which means it is ready for the discretization. 

\subsection{The spatial discretization}
A standard centered finite difference approximation in space is employed. For simplicity, we consider a square computational domain and a uniform spatial mesh: $\Omega = (0,L)^2$ and $h=L/N$, where $N$ is a given integer. The notation for the three-dimensional case can be extended naturally.

\begin{definition}[Mesh functions]
	For any positive integer $N$, define the following point sets:
	$$
	\mathcal{E}_{N}:=\{i \cdot h \mid i=0, \ldots, N \}, \quad \mathcal{C}_{N}:=\{(i-1 / 2) \cdot h \mid i=1, \ldots, N \},
	$$
	$$
	\mathcal{C}_{\bar{N}}:=\{(i-1 / 2) \cdot h \mid i=0, \ldots, N+1\}.
	$$
	The two points belonging to $\mathcal{C}_{\bar{N}} \setminus \mathcal{C}_{N}$ are the so-called ghost points, where the function values are determined by the boundary conditions.
	
	Define the function spaces
	$$
	\mathcal{C}_{\Omega}:=\{\phi: \mathcal{C}_{\bar{N}} \times \mathcal{C}_{\bar{N}} \rightarrow \mathbb{R}\}, \quad \mathcal{V}_{\Omega}:=\{\phi: \mathcal{E}_{N} \times \mathcal{E}_{N} \rightarrow \mathbb{R}\}
	$$
	$$
	\mathcal{E}_{\Omega}^{x}:=\{\phi: \mathcal{E}_{N} \times \mathcal{C}_{N}  \rightarrow \mathbb{R}\},\quad
	\mathcal{E}_{\Omega}^{y}:=\{\phi: \mathcal{C}_{N} \times \mathcal{E}_{N} \rightarrow \mathbb{R}\},\quad
	\mathcal{E}_{\Omega}:=\mathcal{E}_{\Omega}^{x} \times \mathcal{E}_{\Omega}^{y}.
	$$
\end{definition}
The functions in $\mathcal{C}_{\Omega}$ are called cell-centered functions. In component form, cell-centered functions are denoted by $\phi_{i,j}:=\phi(\xi_i, \xi_j)$, where $\xi_i:=(i-\frac12)\cdot h$. The functions in $\mathcal{E}_{\Omega}^{x}$, etc., are called face-centered functions. In component form, face-centered functions are denoted by $f_{i+\frac{1}{2},j}:= f(\xi_{i+\frac{1}{2}},\xi_j)$, etc.  
This staggered grid is also known as the marker-and-cell (MAC) grid and was first
proposed for the incompressible Navier-Stokes equations \cite{MAC-1965}. 
These different types of discrete functions can be generated from one another through average operators. 

\begin{definition}[Average operators]
	Define $A_x: \mathcal{C}_{\Omega} \rightarrow \mathcal{E}_{\Omega}^{x}$ and $a_x: \mathcal{E}_{\Omega}^{x} \rightarrow \mathcal{C}_{\Omega}$ by
	$$
	A_x \phi_{i+\frac{1}{2},j}:=\frac{1}{2}(\phi_{i+1,j}+\phi_{i,j}), \quad  a_x f_{i,j}:=\frac{1}{2}(f_{i+\frac{1}{2},j}+f_{i-\frac{1}{2},j}).
	$$
	The operators $A_y$ and $a_y$ are defined analogously, and further lead to the operator $A_h = (A_x,A_y)^T$.
	
	Define $A_{xy}: \mathcal{C}_{\Omega} \rightarrow \mathcal{V}_{\Omega}$ and $a_{xy}: \mathcal{V}_{\Omega} \rightarrow \mathcal{C}_{\Omega}$ by
	\begin{align}
		& A_{xy}\phi_{i+\frac{1}{2},j+\frac{1}{2}}:=\frac{1}{4}(\phi_{i+1,j}+\phi_{i,j}+\phi_{i+1,j+1}+\phi_{i,j+1}), \nonumber \\
		& a_{xy}f_{i,j}:=\frac{1}{4}( f_{i+\frac{1}{2},j+\frac{1}{2}}+f_{i-\frac{1}{2},j+\frac{1}{2}}+f_{i+\frac{1}{2},j-\frac{1}{2}}+f_{i-\frac{1}{2},j-\frac{1}{2}} ).\nonumber
	\end{align}
	
\end{definition}

The discrete boundary conditions associated with cell-centered and edge-centered functions are introduced in the following definition.
\begin{definition}[Boundary conditions]
	A discrete function $\phi \in \mathcal{C}_{\Omega}$ is said to satisfy homogeneous Neumann boundary conditions, written as $\boldsymbol{\nu} \cdot \nabla_{h} \phi=0$, if $\phi$ satisfies
	$$
	\begin{array}{ll}
		{\phi_{0,j}=\phi_{1,j},} & {\phi_{N,j}=\phi_{N+1,j},} 
		\\
		{\phi_{i,0}=\phi_{i,1},} & {\phi_{i,N}=\phi_{i,N+1}.} 
	\end{array}
	$$
	
	A discrete function $\phi \in \mathcal{C}_{\Omega}$ is said to satisfy Dirichlet boundary conditions, written as $\phi|_\Gamma=\phi_b$, if $\phi$ satisfies
	$$
	\frac{1}{2}(\phi_{0,j}+\phi_{1,j})=\frac{1}{2}(\phi_{N,j}+\phi_{N+1,j})=\frac{1}{2}(\phi_{i,0}+\phi_{i,1})=\frac{1}{2}(\phi_{i,N}+\phi_{i,N+1})=\phi_b.
	$$
	
	A discrete function $\boldsymbol{f}=(f^{x}, f^{y})^{T} \in \mathcal{E}_{\Omega}$ is said to satisfy homogeneous Dirichlet boundary conditions, or simply zero boundary conditions, if
	$$
	\begin{array}{llll}
		{f_{1 / 2, j}^{x}=0,}&{f_{N+1 / 2, j}^{x}=0,}& {f_{i+1/2,0}^{x}+f_{i+1/2,1}^{x}=0,} & {f_{i+1/2,N}^{x}+f_{i+1/2,N+1}^{x}=0,} 
		\\
		{f_{i, 1/2}^{y}=0,}&{f_{i, N+1/2}^{y}=0,}& {f_{0,i+1/2}^{y}+f_{1,i+1/2}^{y}=0,} & {f_{N,i+1/2}^{y}+f_{N+1,i+1/2}^{y}=0.} 
	\end{array}
	$$
\end{definition}

The centered difference operators are defined as follows.
\begin{definition}[Difference operators]
	Define $d_{x}: \mathcal{E}_{\Omega}^{x} \rightarrow \mathcal{C}_{\Omega}$ component-wise by
	\[
	d_{x} f_{i,j}:=\frac{1}{h}\left(f_{i+\frac{1}{2}, j}-f_{i-\frac{1}{2},j}\right),
	\]
	with $d_{y}: \mathcal{E}_{\Omega}^{y} \rightarrow \mathcal{C}_{\Omega}$ defined analogously. Then we define the discrete divergence by
	\[
	\nabla_{h} \cdot: \mathcal{E}_{\Omega} \rightarrow \mathcal{C}_{\Omega}, \quad\quad \nabla_{h} \cdot \boldsymbol{f}:=d_{x} f^{x}+d_{y} f^{y},
	\]
	where $\boldsymbol{f}=(f^x, f^y)^T \in \mathcal{E}_{\Omega}$.
	
	Define $D_{x}: \mathcal{C}_{\Omega} \rightarrow \mathcal{E}_{\Omega}^{x}$ component-wise by
	$$D_{x} \phi_{i+\frac{1}{2},j}:=\frac{1}{h}\left(\phi_{i+1,j}-\phi_{i, j}\right),$$
	with $D_{y}: \mathcal{C}_{\Omega} \rightarrow \mathcal{E}_{\Omega}^{y}$ defined analogously. Then the discrete gradient is given by
	$$\nabla_{h}: \mathcal{C}_{\Omega} \rightarrow \mathcal{E}_{\Omega},\quad\quad \nabla_{h} \phi:=\left(D_{x} \phi, D_{y} \phi \right)^{T} .$$
	
	The standard discrete Laplace operator is
	$$\Delta_{h}: \mathcal{C}_{\Omega} \rightarrow \mathcal{C}_{\Omega}, \quad\quad \Delta_{h} \phi:=\nabla_{h} \cdot \nabla_{h} \phi.$$
\end{definition}

We are now ready to define the following grid inner products and norms.

\begin{definition}[Inner products]
	$$(\phi, \psi):=h^2 \sum_{i=1}^{N} \sum_{j=1}^{N}\phi_{i,j} \psi_{i,j}, \quad\quad\quad\quad \forall \phi, \psi \in \mathcal{C}_{\Omega} , $$
	$$[f, g]_{x}:=\frac{1}{2} h^2 \sum_{i=1}^{N} \sum_{j=1}^{N} (f_{i+\frac{1}{2}, j} g_{i+\frac{1}{2}, j}+f_{i-\frac{1}{2}, j} g_{i-\frac{1}{2},j}), \quad\quad\quad\quad \forall f,g \in \mathcal{E}_{\Omega}^x ,$$
	and $[\cdot,\cdot]_y$ is defined analogously.
	For $\boldsymbol{f}=(f^x,f^y)^T$, $\boldsymbol{g}=(g^x,g^y)^T$, with $\boldsymbol{f},\boldsymbol{g}\in \mathcal{E}_{\Omega}$, define the inner product by
	$$(\boldsymbol{f}, \boldsymbol{g}):=[f^{x}, g^{x}]_{x}+[f^{y}, g^{y}]_{y}. $$
\end{definition}

\begin{definition}[Norms]
	For $\boldsymbol{f}\in \mathcal{E}_{\Omega}$, define the norm induced by the inner product by
	$$\|\boldsymbol{f}\|_2=\sqrt{(\boldsymbol{f},\boldsymbol{f})}.$$
	
	In addition, for $\phi\in\mathcal{C}_{\Omega}$, define
	$$\|\phi\|_{\infty}:=\max _{i,j}|\phi_{i,j}|,$$
	$$\|\phi\|_p:=(|\phi|^p,1)^{\frac{1}{p}},\quad\quad 1\leq p < \infty,$$
	$$\left\|\nabla_{h} \phi\right\|_{p}:=\left(\left[\left|D_{x} \phi\right|^{p}, 1\right]_{x}+\left[\left|D_{y} \phi\right|^{p}, 1\right]_{y}\right)^{\frac{1}{p}}.$$
	It is obvious that $(\nabla_{h} \phi,\nabla_{h} \phi)=\|\nabla_{h} \phi\|_2^2$ when $p=2$.
\end{definition}

For $\phi,\psi \in \mathcal{C}_{\Omega}$ and vector fields $\bof \in \mathcal{E}_{\Omega}$, the following summation-by-parts formulas can be derived through standard calculations.
\begin{lem}[Summation-by-parts formulas]
	Suppose $\phi,\psi \in \mathcal{C}_{\Omega}$ and the velocity vector field $\boldsymbol{f} \in \mathcal{E}_{\Omega}$. Then we have:
	
	(1) If $\psi$ satisfies the homogeneous Neumann boundary condition $\boldsymbol{\nu} \cdot \nabla_{h} \psi=0$, then 
	$$(\phi,\nabla_h\cdot(\bof \nabla_h\psi))=-(\nabla_h\phi,\bof \nabla_h\psi).$$
	
	(2) If $\boldsymbol{f}$ satisfies the homogeneous Dirichlet boundary condition, then
	$$(\phi, \nabla_h\cdot\boldsymbol{f})=-(\nabla_h\phi, \boldsymbol{f}).$$
	To avoid ambiguity, we clarify the product without symbol: 
    $$\boldsymbol{fg}:=(f^xg^x,f^yg^y)^T, \quad \text{for } \boldsymbol{f}=(f^x,f^y)^T, \boldsymbol{g}=(g^x,g^y)^T \in \mathcal{E}_{\Omega}. $$
\end{lem}

\subsection{The decoupled component-wise convex-splitting scheme}
The idea to develop an entropy-stable scheme is based on the component-wise convex-splitting method \cite{Kou-18,Kou-18b}. The entropy without the gradient contribution is given by
\begin{align*}
	s(n,e)=-n\left(\frac{3}{2}\ln n+\ln\frac{n}{1-n}+M\right)+\frac{3}{2}n\ln(e+n^2)=:f(n)+g(n,e).
\end{align*}
By physical definition, the number density $0<n<1$, and the internal energy $e$ is positive. This leads to a component-wise convex-concave decomposition of the entropy function: 
\begin{align}
	f(n)=-n\left(\frac{3}{2}\ln n+\ln\frac{n}{1-n}+M\right), \quad g(n,e)=\frac{3}{2}n\ln(e+n^2)\nonumber
\end{align}
where $f$ is concave w.r.t. $0<n<1$, $g$ is concave w.r.t. $e>0$ and convex with respect to $n>0$.
For the discrete variational derivative, a component-wise procedure is considered:
\begin{align*}
	\frac{s(n^{k+1},e^{k+1})-s(n^{k},e^{k})}{\dt} &= \frac{s(n^{k+1},e^{k})-s(n^{k},e^{k})}{\dt} +\frac{s(n^{k+1},e^{k+1})-s(n^{k+1},e^{k})}{\dt} \\
	&=:I+II.
\end{align*}
By the convexity of functions $f$ and $g$, we have :
\begin{align}
	& I \geq \frac{\partial s^{k+1}}{\partial n} \frac{n^{k+1}-n^{k}}{\dt}, \quad II \geq \frac{\partial s^{k+1}}{\partial e} \frac{e^{k+1}-e^{k}}{\dt}, \label{convex 1} \\
	& \frac{\partial s^{k+1}}{\partial n}:=f^\prime(n^{k+1})+\frac{\partial g}{\partial n}(n^{k},e^{k}), \quad \frac{\partial s^{k+1}}{\partial e}:=\frac{\partial g}{\partial e}(n^{k+1},e^{k+1})=\frac{1}{T^{k+1}}.\label{convex 2}
\end{align}
\eqref{convex 1} and \eqref{convex 2} lead to the convex-splitting scheme where the convex part is updated explicitly, while the concave part is treated implicitly. Given $n^k>0$, $e^k>0$, and $\bu^k$, solve for $n^{k+1}$, $e^{k+1}$, $T^{k+1}$, and $\bu^{k+1}$ such that
\begin{align}
	& \frac{n^{k+1} - n^k}{\dt} + \nabla_h \cdot (A_h n^{k+1} \bu^k) = 0, \label{decoupled scheme 1} \\
	& \frac{e^{k+1} - e^k}{\dt} + T^{k+1} \nabla_h \cdot \left( A_h \hat{s}^{k} \bu^k \right) + \hat{\mu}^{k+1} \nabla_h \cdot \left( A_h n^{k+1} \bu^k \right)      \label{decoupled scheme 2} \\
	&\quad\quad\quad = R \bsigma^{k}_* :\nabla_h \bu^{k} +\nabla_h \cdot \left( A_h n^k \nabla_h T^{k+1} \right) \nonumber \\
    &\quad\quad\quad\quad\quad\quad - 2 T^{k+1} \nabla_h \cdot \left( \nabla_h n^{k+1}\otimes\nabla_h n^{k+1} \cdot \bu^k \right) ,\nonumber \\
	& \hat{\mu}^{k+1} = -T^{k+1}\left( f^\prime(n^{k+1})+\frac{\partial g}{\partial n}(n^{k},e^{k}) + 2 \Delta_h n^{k+1} \right), \label{decoupled scheme 2.1} \\
	& T^{k+1} = \frac{2}{3}\frac{e^{k+1} + (n^{k+1})^2}{n^{k+1}} , \quad \bsigma^{k}_* = n^{k+1} \left( \nabla_h \bu^{k} + (\nabla_h \bu^{k})^T + \frac{1}{3}\nabla_h \cdot \bu^{k}\bI\right), \label{decoupled scheme 3} \\
	& R A_hn^{k+1} \left( \frac{\bu^{k+1} - \bu^k}{\dt} + \bu^k \cdot \nabla_h \bu^{k+1} \right) = R \nabla_h \cdot \bsigma^{k+1}- G A_y n^{k+1} \be_y  \label{decoupled scheme 4} \\
	& \quad\quad - \left( A_h n^{k+1}\nabla_h \hat{\mu}^{k+1} + A_h\hat{s}^{k}\nabla_h T^{k+1} + 2\nabla_h n^{k+1} \otimes \nabla_h n^{k+1}\cdot \nabla_h T^{k+1} \right), \nonumber \\
	& \bsigma^{k+1} = n^{k+1} \left( \nabla_h \bu^{k+1} + (\nabla_h \bu^{k+1})^T+ \frac{1}{3}\nabla_h \cdot \bu^{k+1}\bI \right). \label{decoupled scheme 5} 
\end{align}
The boundary conditions are $\bu^{k+1} |_\Gamma = \boldsymbol{0}$ and $\bnu\cdot \nabla_h n^{k+1} = 0$. The temperature may satisfy either $\bnu\cdot\nabla_h T^{k+1} = 0$ or the Dirichlet boundary condition $T|_\Gamma=T_b$. The boundary condition of $e^{k+1}$ and $\hat{\mu}^{k+1}$ is determined by $n^{k+1}$ and $T^{k+1}$. In detail, for the centered difference scheme on staggered grids, some interpolations are required in the tensor computations: given $\bu=(u^x,u^y)$, $\bv=(v^x,v^y)$, $\bof=(f^x,f^y) \in \mathcal{E}_{\Omega}$ 
\begin{align}
	& \bu \cdot \nabla_h \bv :=\binom{\ u^x \cdot D_x a_x(v^x)+A_{x y} u^y \cdot d_y A_y(v^x)\ }{\ A_{x y} u^x \cdot d_x A_x(v^y)+ u^y \cdot D_y a_y(v^y)\ }, \label{interpolations 1} \\
	& \nabla_h n^{k+1} \otimes \nabla_h n^{k+1} \cdot \bof :=A_h \left(a_x(D_x n^{k+1}f^x)+a_y(D_yn^{k+1}f^y)\right) \nabla_h n^{k+1}, \label{interpolations 2} \\
	& \bsigma^{k}_* :\nabla_h \bu^{k}:=2n^{k+1}\left( (d_x u^{x,k})^2+(d_y u^{y,k})^2  \right)+n^{k+1}a_{xy}\left( D_x u^{y,k} + D_y u^{x,k}  \right)^2 \label{interpolations 3} \\
	& \quad\quad\quad\quad\quad\quad+ \frac{1}{3}n^{k+1}(\nabla_h \cdot \bu^k)^2, \nonumber \\
	& \nabla_h \cdot \bsigma^{k+1} := \frac{1}{3}\nabla_h ( n^{k+1}\nabla_h\cdot\bu^{k+1} ) \label{interpolations 4} \\
    & \quad +\binom{\  2D_x\left( n^{k+1}a_x(d_x u^{x,k+1}) \right)+d_y \left( A_{xy}n^{k+1}( D_x u^{y,k+1}+D_y u^{x,k+1} ) \right) \ }{\ d_x \left( A_{xy}n^{k+1}( D_x u^{y,k+1}+D_y u^{x,k+1} )\right) + 2D_y\left( n^{k+1}a_y(d_y u^{y,k+1}) \right) \ } \nonumber 
\end{align}

In implementation, the numerical scheme is divided into three steps: (i) solve the mass equation \eqref{decoupled scheme 1} for $n^{k+1}$; (ii) solve the energy equations \eqref{decoupled scheme 2}-\eqref{decoupled scheme 3} for $e^{k+1}$ and $T^{k+1}$; (iii) solve the momentum equations \eqref{decoupled scheme 4}-\eqref{decoupled scheme 5} for the velocity $\bu^{k+1}$. All three parts are linear. The key feature is the decoupled framework. Since the number density is determined from the mass conservation equation before solving the energy equation, the logarithmic term does not enter a coupled nonlinear solve. The linear dependence of temperature on internal energy holds once the density is fixed. In addition, the momentum equation does not directly affect the entropy analysis, which makes a fully decoupled numerical scheme possible. Considering the numerical stability, the semi‑implicit method is applied for the velocity. 
It is worth to mention that that, for a prescribed velocity field, solving the mass conservation equation with central differences and a fully implicit time discretization might lead to numerical oscillations near the interface. In the proposed scheme \eqref{decoupled scheme 1}-\eqref{decoupled scheme 5}, although the mass equation is solved in a decoupled manner, the stability induced by entropy production prevents such numerical oscillations from appearing in the simulations reported below. 
In addition, since the logarithmic term is treated essentially explicitly, the positivity analysis cannot be established. This is a trade-off inherent in the construction of a linear scheme and will be left for future work.

Mass conservation follows naturally from the summation‑by‑parts identity, which is made possible by the homogeneous Dirichlet boundary condition imposed on velocity. Therefore, we do not provide a separate proof. The theorem below demonstrates that the scheme satisfies a discrete entropy-production inequality without a time-step restriction. The argument is mainly based on the convex–concave splitting of the entropy function. The tensor calculus \eqref{interpolations 1}-\eqref{interpolations 4} on the staggered grid also supply the appropriate boundary conditions and positivity required for summation‑by‑parts formula. 
\begin{theorem}\label{thm:entropy}
	The homogeneous Neumann conditions are enforced for density and internal energy, and homogeneous Dirichlet conditions for velocity. Ignore the gravity as $G=0$. Assume that the numerical solutions $n^{k+1}$ and $e^{k+1}$ are positive point-wisely, then the proposed scheme \eqref{decoupled scheme 1}-\eqref{decoupled scheme 5} satisfies the entropy production law
	\begin{align}
		\frac{\hat{S}^{k+1}-\hat{S}^{k}}{\dt}  \geq \left( \frac{R}{T^{k+1}}, \bsigma^{k}_* :\nabla_h \bu^{k} \right) -  \left(  \nabla_h \frac{1}{T^{k+1}}, A_h n^k\nabla_h T^{k+1} \right) \geq 0, \label{entropy production 2}
	\end{align}
	where 
	\begin{equation}
		\hat{S}^k = (s^k,1)-\| \nabla_h n^k \|^2.\nonumber
	\end{equation}
\end{theorem}
\begin{proof}
	Based on the idea of the convex-splitting method, we calculate the difference of the entropy between two adjacent time step:
	\begin{align}
		&\frac{\hat{S}^{k+1}-\hat{S}^{k}}{\dt} \geq \left( \frac{e^{k+1}-e^{k}}{\dt}, \frac{1}{T^{k+1}} \right) \nonumber \\
        & \quad\quad\quad\quad  +\left( \frac{n^{k+1}-n^{k}}{\dt}, \frac{\partial s^{k+1}}{\partial n} \right) + \left( \frac{n^{k+1}-n^{k}}{\dt}, 2\Delta_h n^{k+1} \right). \nonumber
	\end{align}
	The mass equation and the energy equation imply that
	\begin{align}
		&\frac{\hat{S}^{k+1}-\hat{S}^{k}}{\dt} \geq \left( \nabla_h\cdot(A_hn^{k+1} \bu^k), \frac{\hat{\mu}^{k+1}}{T^{k+1}} \right)  - \left( \nabla_h \cdot \left( A_h n^{k+1} \bu^k \right),\frac{\hat{\mu}^{k+1}}{T^{k+1}} \right) \nonumber \\
        & \quad\quad\quad\quad\quad - \left( \nabla_h \cdot (A_h\hat{s}^k \bu^k),1 \right) -2\left( \nabla_h \cdot \left( \nabla_h n^{k+1}\otimes\nabla_h n^{k+1} \cdot \bu^k \right),1 \right) \nonumber\\
		& \quad\quad\quad\quad\quad + R \left( \frac{1}{T^{k+1}}, \bsigma^{k}_* :\nabla_h \bu^{k} \right) +  \left( \frac{1}{T^{k+1}}, \nabla_h\cdot(A_h n^k\nabla_h T^{k+1}) \right) \nonumber \\
		& \quad=  \left( \frac{R}{T^{k+1}}, \bsigma^{k}_* :\nabla_h \bu^{k} \right) +  \left( \frac{1}{T^{k+1}}, \nabla_h\cdot(A_h n^k\nabla_h T^{k+1}) \right)  \nonumber \\
		& \quad= \left( \frac{R}{T^{k+1}}, \bsigma^{k}_* :\nabla_h \bu^{k} \right) -  \left(  \nabla_h \frac{1}{T^{k+1}}, A_h n^k\nabla_h T^{k+1} \right).\nonumber
	\end{align}
	Here, the summation-by-parts formula under the homogeneous Neumann boundary condition and zero boundary condition is applied. In particular, the stress operation $\nabla_h n^{k+1}\otimes\nabla_h n^{k+1} \cdot \bu^k$ is designed to satisfy the Neumann boundary condition based in the definition \eqref{interpolations 2}. A further calculation shows the positivity:
	\begin{align}
		 -[D_x\frac{1}{T^{k+1}}, A_x n^k D_x T^{k+1}]_{x} =& \frac{1}{2} \sum_{i,j=1}^{N} \frac{(T_{i+1,j}^{k+1}-T_{i,j}^{k+1})^2}{T_{i+1,j}^{k+1}T_{i,j}^{k+1}} A_x n^{k}_{i+\frac{1}{2}, j} \label{gradient 1} \\
        & +\frac{1}{2} \sum_{i,j=1}^{N} \frac{(T_{i-1,j}^{k+1}-T_{i,j}^{k+1})^2}{T_{i-1,j}^{k+1}T_{i,j}^{k+1}} A_x n^{k}_{i-\frac{1}{2}, j} \geq 0, \nonumber \\
		 -[D_y\frac{1}{T^{k+1}}, A_y n^k D_y T^{k+1}]_{y} = &\frac{1}{2} \sum_{i,j=1}^{N} \frac{(T_{i,j+1}^{k+1}-T_{i,j}^{k+1})^2}{T_{i,j+1}^{k+1}T_{i,j}^{k+1}} A_y n^{k}_{i, j+\frac{1}{2}} \label{gradient 2} \\
        &  +\frac{1}{2} \sum_{i,j=1}^{N}\frac{(T_{i,j-1}^{k+1}-T_{i,j}^{k+1})^2}{T_{i,j-1}^{k+1}T_{i,j}^{k+1}} A_y n^{k}_{i, j-\frac{1}{2}} \geq 0. \nonumber 
    \end{align}
    The inequality \eqref{gradient 1} and \eqref{gradient 2} lead to the required positivity
    \begin{align}
        -  \left(  \nabla_h \frac{1}{T^{k+1}}, A_h n^k\nabla_h T^{k+1} \right) \geq 0. \nonumber
    \end{align}
	Consequently, the entropy production law \eqref{entropy production 2} is proved.
\end{proof}

\begin{remark}
	For the dynamic van der Waals model, the entropy production can be fulfilled by the classical methods for phase-field models. However, energy conservation at the fully discrete level is more challenging. There are two issues. The temporal issue arises from the decoupling strategy: the velocity treated explicitly in the mass equation is inconsistent with that in the momentum equation. As a result, it is difficult for decoupled schemes to preserve energy conservation, even in the semi-discrete sense. A semi-discrete scheme was designed in \cite{Kou-18b}. Although energy conservation is maintained via artificial dissipation, a decoupled solution strategy is not feasible. 
	Spatially, the difficulty stems from the centered difference discretization on staggered grids. The introduction of averaging operators destroys the key integration-by-parts identities and the chain rules, which is essential. To the best of our knowledge, no numerical work has yet been reported that can preserve both energy conservation and entropy production in the fully discrete level.
\end{remark}

\section{Numerical simulations} \label{sec: simulations}
In this section, the convergence test of the numerical scheme and simulations of non-isothermal single-component two-phase flow problems are proposed. In the first, we denote the critical number density and critical temperature as $n_c$ and $T_c$, respectively. The theoretical values are calculated in \cite{Onuki-PRE-05,Onuki-PRE-07} as $n_c=\frac{1}{3v_0}$ and $T_c=\frac{8\varepsilon}{27k_B}$ which lead to the non-dimensional version $n_c=\frac{1}{3}$, $T_c=\frac{8}{27}$ (use the same notations without ambiguity) under the scale described in Section \ref{sec:dimensionless}. The initial data for the following numerical simulations will be set according to these reference values. In the convergence test, the choice of the initial values \eqref{initial n}-\eqref{initial u} also takes into account the magnitudes of the physical quantities.

\subsection{Convergence test}
In this part, we use smooth initial data to test the convergence rates of the numerical scheme \eqref{decoupled scheme 1}-\eqref{decoupled scheme 5}, with no gravity ($G=0$), on the isolated domain $\Omega=(0,L)^2$, where $L=100$ and $R=0.5$. The initial data are chosen to be trigonometric:
\begin{align}
	& n^0(x,y)=(\cos(2\pi y/100)\cos(2\pi x/100)+2)/18, \label{initial n}\\
	& T^0(x,y)=(\cos(2\pi y/100)\cos(2\pi x/100)+2)/18+0.2, \label{initial T}\\
	& \bu^0 = \boldsymbol{0}. \label{initial u}
\end{align}
Based on the decoupling property and the use of centered differences, the numerical algorithm is expected to be first-order accurate in time and second-order accurate in space. As the exact solution is unavailable, the Cauchy difference on adjacent grids is employed to calculate the temporal and spatial convergence orders, respectively. The definition is given as 
\begin{equation}
	\delta_{\phi} := \phi_{h_c} - I_c^f (\phi_{h_f}), \quad \phi = n \text{ or } T \nonumber
\end{equation}
where $h_f$ denotes the numerical solution on the fine spatial (or temporal) mesh and $h_c$ is the one on the coarse spatial (or temporal) mesh. Here we set $h_c=2h_f$. In fact, $I_c^f$ is a bilinear projection, mapping the solution from fine grid to coarse grid. 
For the spatial order, we fix $\dt=0.01$ and change space step as $h=L/N$, $N\in [16,32,64,128,256,512]$. The Cauchy differences at final time $T_{final}=10$ and corresponding numerical convergence orders are presented in Tables \ref{space order}, both of which are consistent with the expected results.
\begin{table}[t]
	\centering
	\begin{tabular}{llllll}
		\hline
		\textbf{Density}     & 32-16     & 64-32     & 128-64    & 256-128   & 512-256   \\ \hline
		$\ell^2$ error     & 2.667E-02 & 6.668E-03 & 1.654E-03 & 4.101E-04 & 1.022E-04 \\
		$\ell^2$ order     &           & 2.000     & 2.011     & 2.012     & 2.005     \\
		$\ell^\infty$ error & 5.136E-04 & 1.325E-04 & 3.339E-05 & 8.365E-06 & 2.092E-06 \\
		$\ell^\infty$ order &           & 1.955     & 1.989     & 1.997     & 1.999     \\ \hline
	\end{tabular}
	
	\vspace{5mm}
	\begin{tabular}{llllll}
		\hline
		\textbf{Temperature}      & 32-16     & 64-32     & 128-64    & 256-128   & 512-256   \\ \hline
		$\ell^2$ error     & 2.661E-02 & 6.619E-03 & 1.618E-03 & 3.958E-04 & 9.818E-05 \\
		$\ell^2$ order     &           & 2.007     & 2.032     & 2.031     & 2.011     \\
		$\ell^\infty$ error & 5.158E-04 & 1.331E-04 & 3.355E-05 & 8.403E-06 & 2.102E-06 \\
		$\ell^\infty$ order &           & 1.954     & 1.988     & 1.997     & 1.999     \\ \hline
	\end{tabular}
    \vspace{5mm}
	\caption{The spatial numerical convergence order of the number density $n$ (top) and the temperature $T$ (bottom), with the initial data \eqref{initial n}-\eqref{initial u}, fixed time step $\dt=0.01$ and final time $T_{final}=10$.}\label{space order}
\end{table}

As for the temporal order, the space step is fixed as $h=L/256$ and the time steps are chosen as $\dt=1/S$, $S \in [1,2,4,8,16,32]$. The final time remains as $T_{final}=10$. The results are shown in Table \ref{time order}. This verifies the first-order convergence in time.
\begin{table}[t]
	\centering
	\begin{tabular}{llllll}
		\hline
		\textbf{Density}        & 20-10         & 40-20     & 80-40     & 160-80    & 320-160   \\ \hline
		$\ell^2$ error         & 1.583E-02     & 8.220E-03 & 4.194E-03 & 2.119E-03 & 1.065E-03 \\
		$\ell^2$ order         &               & 0.945     & 0.971     & 0.985     & 0.992     \\
		$\ell^\infty$ error     & 4.286E-04     & 2.272E-04 & 1.172E-04 & 5.952E-05 & 3.000E-05 \\
		$\ell^\infty$ order  &   & 0.916         & 0.955     & 0.978    & 0.988     \\ \hline
	\end{tabular}
	
	\vspace{5mm}
	\begin{tabular}{llllll}
		\hline
		\textbf{Temperature}    & 20-10         & 40-20     & 80-40     & 160-80    & 320-160   \\ \hline
		$\ell^2$ error         & 4.892E-02     & 2.527E-02 & 1.286E-02 & 6.487E-03 & 3.258E-03 \\
		$\ell^2$ order         &               & 0.953  & 0.975  & 0.987  & 0.994  \\
		$\ell^\infty$ error     & 1.128E-03     & 5.997E-04 & 3.097E-04 & 1.574E-04 & 7.939E-05 \\
		$\ell^\infty$ order   &  & 0.911  & 0.953   & 0.976  & 0.987  \\ \hline
	\end{tabular}
    \vspace{5mm}
	\caption{The temporal numerical convergence order of the number density $n$ (top) and the temperature $T$ (bottom), with the initial data \eqref{initial n}-\eqref{initial u}, fixed space step $h=L/256$ and final time $T_{final}=10$.} \label{time order}
\end{table}
\subsection{Droplet evaporation}
In this test case, we present a benchmark numerical study of droplet evaporation. A relatively large isolated system is used to approximate an open environment.
For the boundary conditions, $n$ and $e$ satisfy homogeneous Neumann boundary conditions, while the velocity satisfies homogeneous Dirichlet boundary conditions. The computational domain is $\Omega = (0,L)^2$ with $L=200$, the time step is $\dt=10^{-2}$, the spatial step is $h=L/256$, and $R=0.02$. Gravity is ignored as $G=0$. The initial data describe a droplet located at the center of the domain with number density $n_l=0.58n_c$ and lower temperature $T_l=0.875T_c$, while the surrounding gas has density $n_g=0.122n_c$ and higher temperature $T_g=T_c$:
\begin{align}
	& f(x,y,a_{out},a_{in})=\frac{a_{out}-a_{in}}{2}\left(\tanh (\sqrt{(x-100)^2+(y-100)^2}-20)+1\right)+a_{in}. \nonumber  \\
	& n^0(x,y) = f(x,y,0.122n_c,0.58n_c), \quad T^0(x,y) = f(x,y,T_c,0.875T_c),\nonumber \\
    & \bu^0 = \boldsymbol{0}.\nonumber
\end{align}
Figure \ref{fig:Droplet n and T} shows the evolution of the number density and temperature. As can be seen, the interface of the droplet progressively thickens and gradually merges into the surrounding environment, while the temperature exhibits similar diffusive behavior. Figure \ref{fig:Droplet entropy} illustrates the preservation of entropy production, mass conservation, and energy conservation.
\begin{figure}[t]
	\centering
	\includegraphics[scale=0.21]{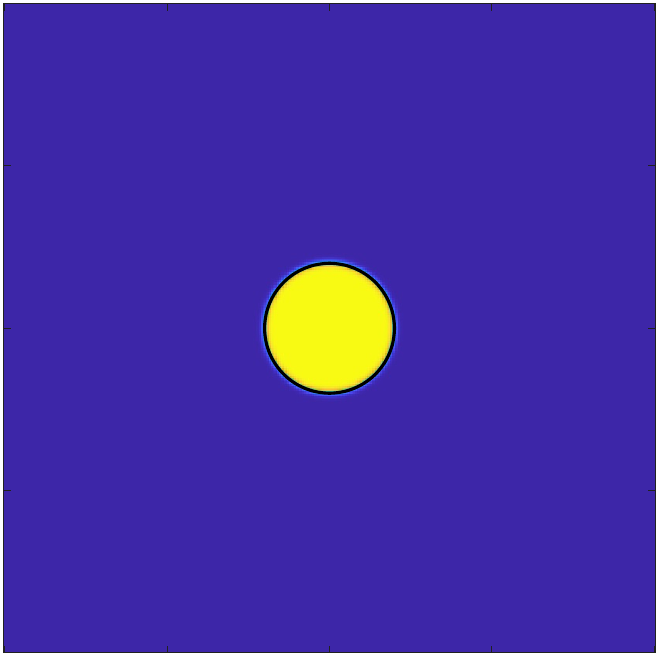}
	\includegraphics[scale=0.21]{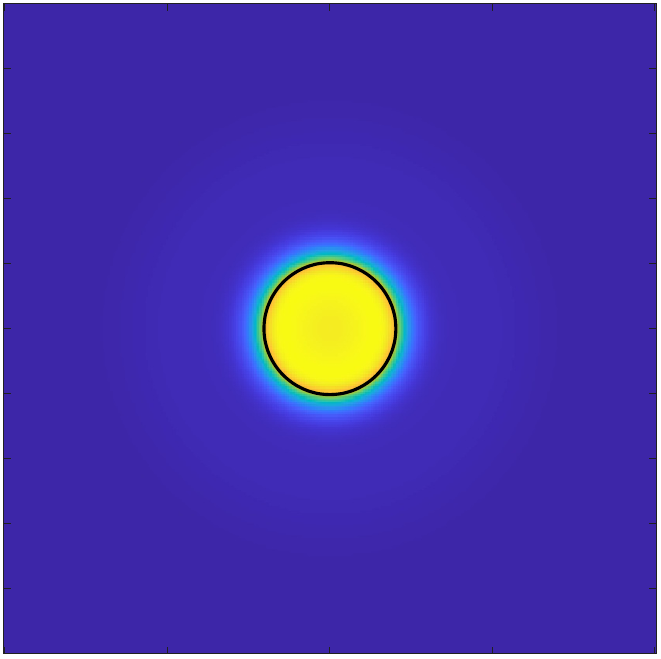}
	\includegraphics[scale=0.21]{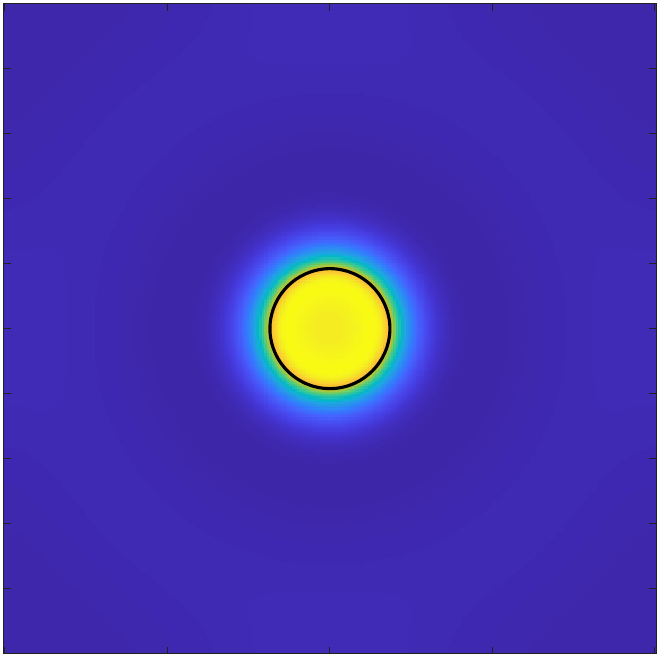}
	\includegraphics[scale=0.21]{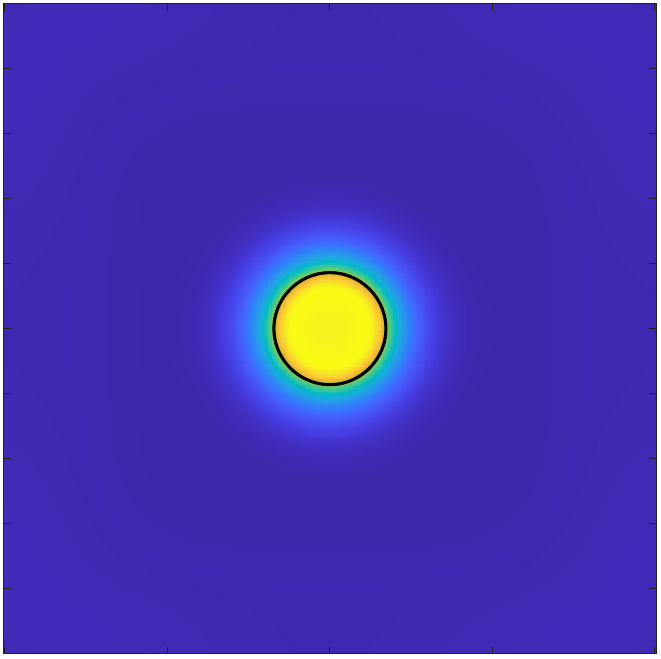}
	\includegraphics[scale=0.21]{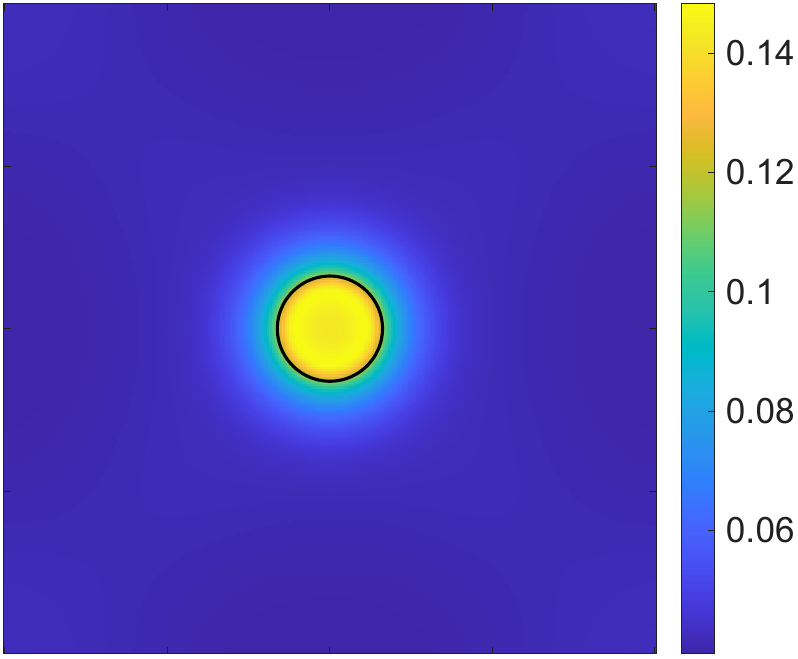}

    \vspace{1mm}
	\hspace{1mm}\includegraphics[scale=0.212]{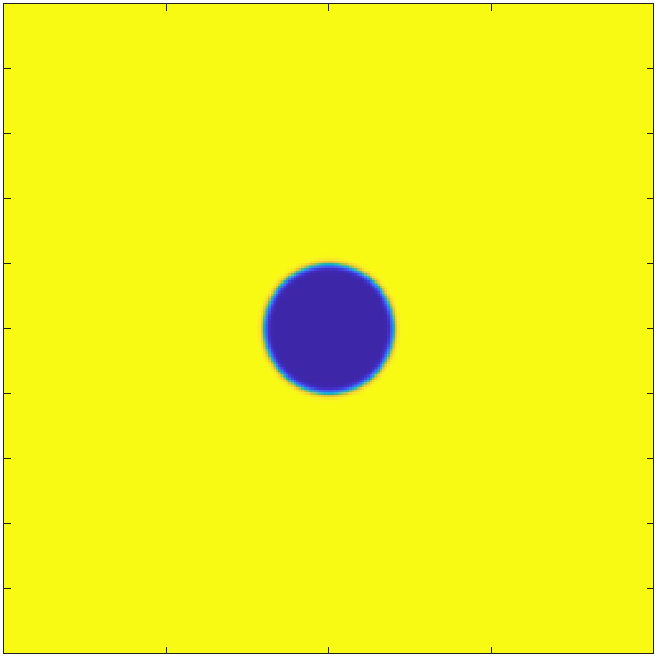}
	\includegraphics[scale=0.212]{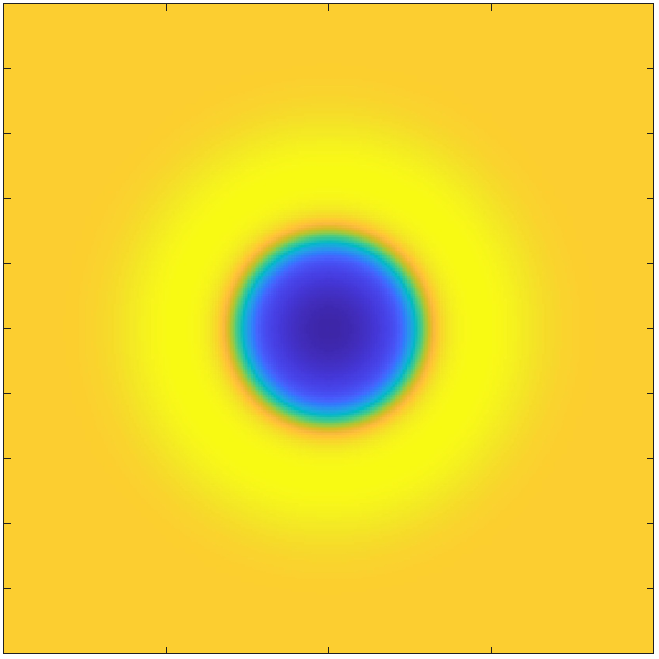}
	\includegraphics[scale=0.212]{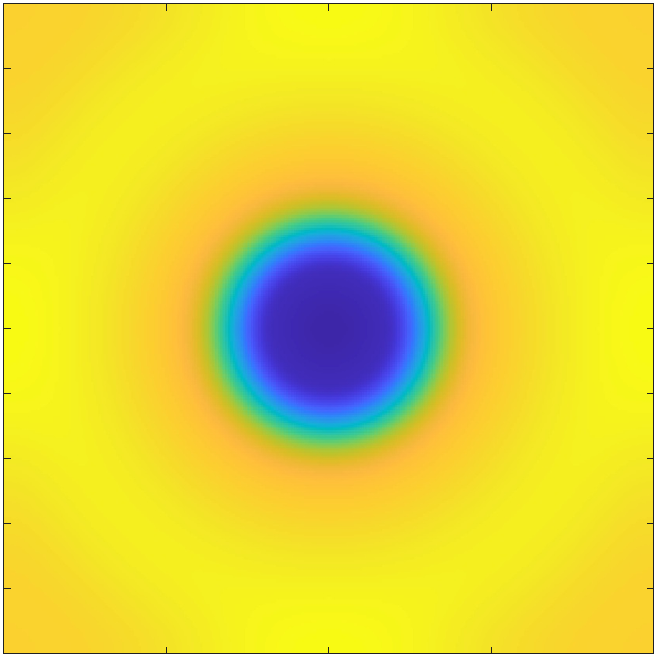}
	\includegraphics[scale=0.212]{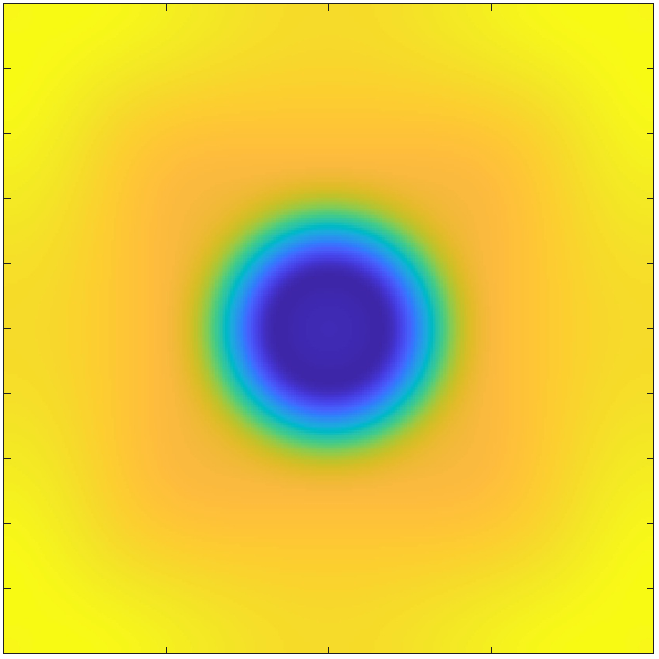}
	\includegraphics[scale=0.212]{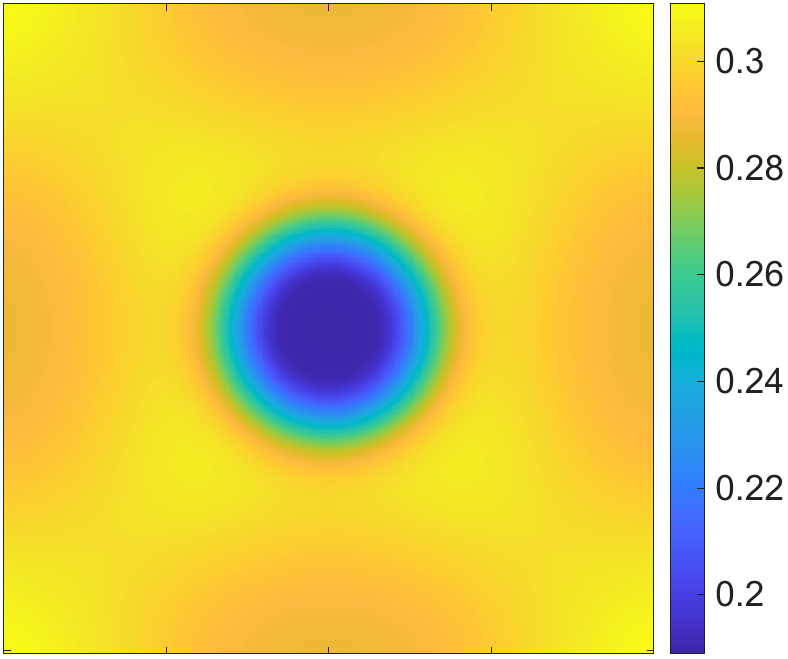}
	\caption{(i) First row: the number density in droplet evaporation. The black circle indicates the interface located by the intermediate number density $n_{mid}=(n_l+n_g)/2 \approx 0.1170$. (ii) Second row: the temperature in droplet evaporation. $t=0,50,100,150,200$.}
	\label{fig:Droplet n and T}
\end{figure}

In addition, the interface location is indicated in the first row of Figure \ref{fig:Droplet n and T} by the intermediate value $n_{mid}=(n_l+n_g)/2 \approx 0.1170$. Over the selected time interval, the numerical results exhibit an approximately linear decrease in $D^2$, consistent with the classical $D^2$-law \cite{law2006combustion,sirignano1999fluid}, which states that the square of the droplet diameter $D$ decreases linearly with time, i.e.
\begin{equation}
	\left(\frac{D}{D_0}\right)^2 = 1-\kappa t,\nonumber
\end{equation}
where $D_0$ is the initial droplet diameter and $\kappa$ is the evaporation constant. Figure \ref{fig:Diameter} provides a quantitative illustration of the squared droplet diameter over time, demonstrating excellent agreement with the $D^2$-law.
\begin{figure}[t]
	\centering
	\includegraphics[scale=0.4]{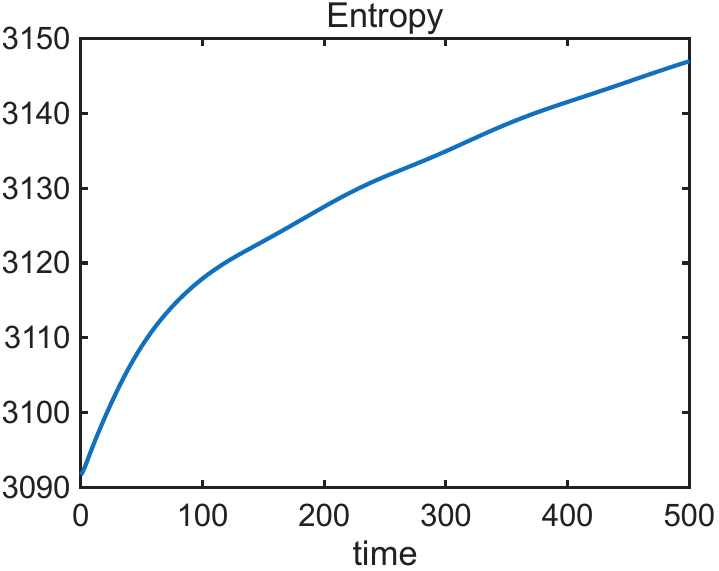}\hspace{10mm}
	\includegraphics[scale=0.4]{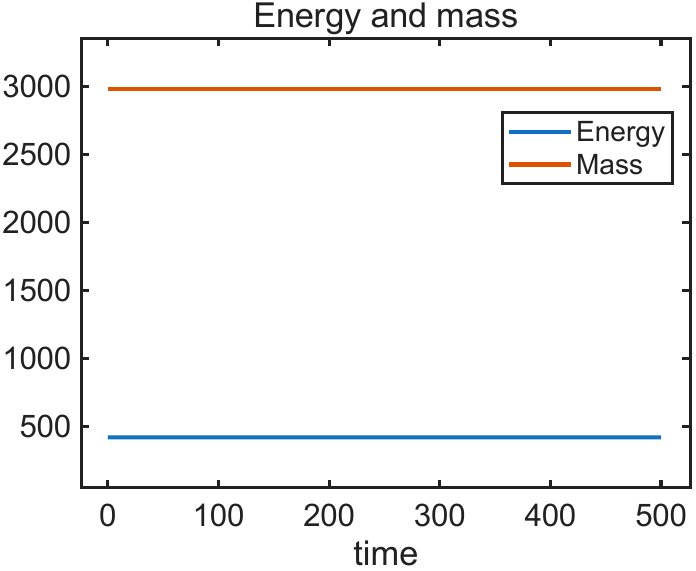}
	\caption{The variations of energy, entropy, and mass in droplet evaporation.}
	\label{fig:Droplet entropy}
\end{figure}

\begin{figure}
	\centering
	\includegraphics[scale=0.45]{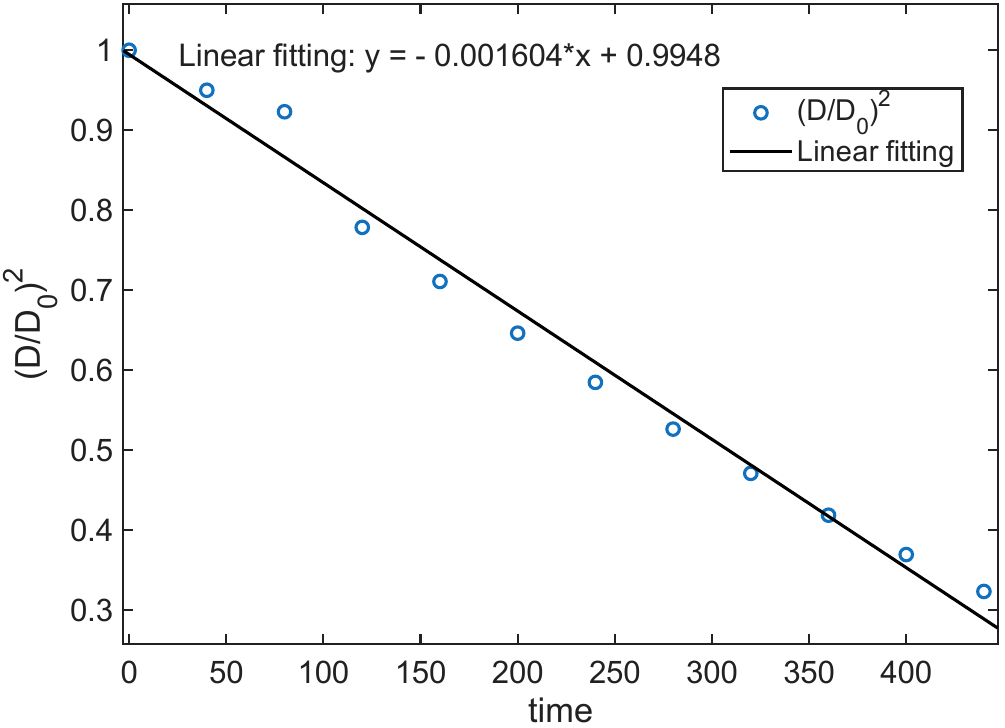}
	\caption{Time evolution of the square of the dimensionless diameter for the present model.}
	\label{fig:Diameter}
\end{figure}

\subsection{Leidenfrost effect in 2D and 3D}
In this section, we simulate the two- and three-dimensional Leidenfrost effect and compare our numerical results with those reported in the literature \cite{Xu-Qian-PRE-13}. The Leidenfrost effect requires that the boundary in contact with the droplet be maintained at a temperature sufficiently higher than the initial droplet temperature to induce rapid evaporation. Therefore, in the simulations, we impose mixed boundary conditions on the temperature: Dirichlet conditions on $y=0,L$ and Neumann conditions on $x=0,L$, while the number density and velocity satisfy homogeneous Neumann and zero boundary conditions, respectively. 
Gravity cannot be neglected in this case. However, the physical quantity in the real world is negligible at the small scales employed in numerical simulations. To make the phenomena more pronounced, the gravitational parameter is artificially amplified and set to $G=10^{-4}$, which is identical with that in \cite{Xu-Qian-PRE-13,Xu-Qian-PRE-14}.
In addition, the coefficient $R$ is chosen as $0.02$, and the computational domain is $\Omega = (0,L)^2$ with $L=200$. The time step is $\dt=10^{-2}$ and the spatial step is $h=L/512$. For the initial data, the temperature at the bottom boundary is higher than that at the top boundary, and the temperature within the region varies linearly:
\begin{align}
	& T^0|_{y=0}=T_h:=0.975T_c, \quad T^0|_{y=L}=T_l:=0.875T_c, \nonumber\\
	& T^0(x,y) = T_h + \frac{T_l-T_h}{L}y.\nonumber
\end{align}
A semi-elliptical droplet adheres to the bottom boundary, with number density $n_l=0.58n_c$ for the liquid and $n_g=0.122n_c$ for the surrounding gas; see Figure \ref{fig:Leidenfrost n 2}.

The temporal evolution of the number density and temperature is presented in Figure \ref{fig:Leidenfrost n 2} and Figure \ref{fig:Leidenfrost T 2}, respectively. It is seen that the droplet is quickly vaporized by the high temperature. The resulting vapor elevates the droplet away from the hot boundary, leading to an indentation on the droplet surface. At the same time, the droplet gradually evaporates and eventually merges with the background gas. These phenomenons are in good agreement with the observations in \cite{Xu-Qian-PRE-13}. Since the droplet is initially at rest on the heated boundary with no initial momentum, it does not exhibit the pronounced vertical bouncing observed in \cite{Qiao_Xu_Yang_Zhang_2026}.
In addition, it can be observed that during levitation, the levitating droplet experiences slight changes in shape and height. These changes, together with the heat exchange (determined by the Dirichlet boundary condition for temperature) and gravity, lead to a fluctuating increase in energy and entropy, which is shown in Figure \ref{fig:Leidenfrost entropy 2}. The mass conservation is represented in Figure \ref{fig:Leidenfrost entropy 2} as well.
\begin{figure}[t]
	\centering
	\includegraphics[scale=0.107]{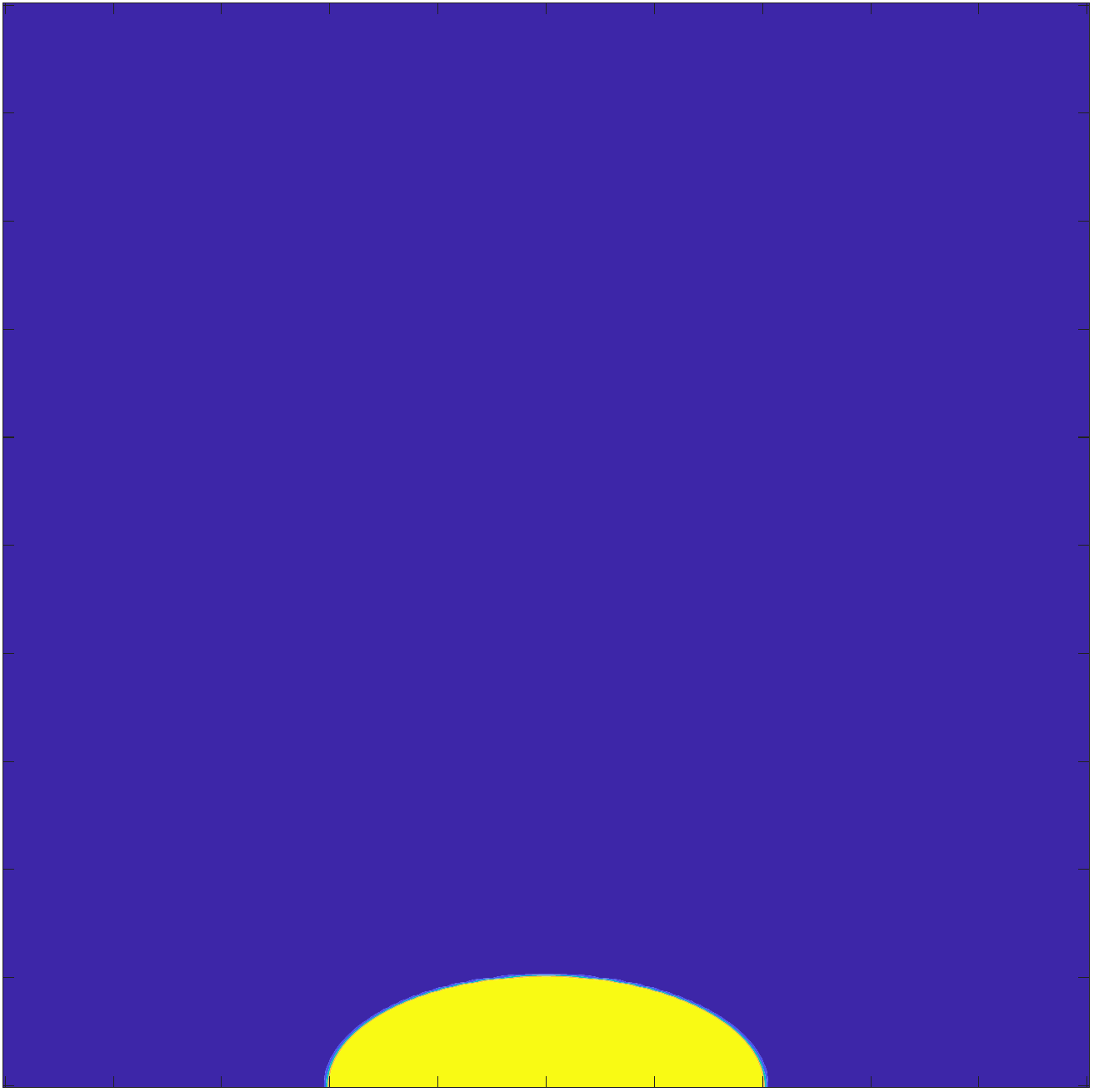}
	\includegraphics[scale=0.107]{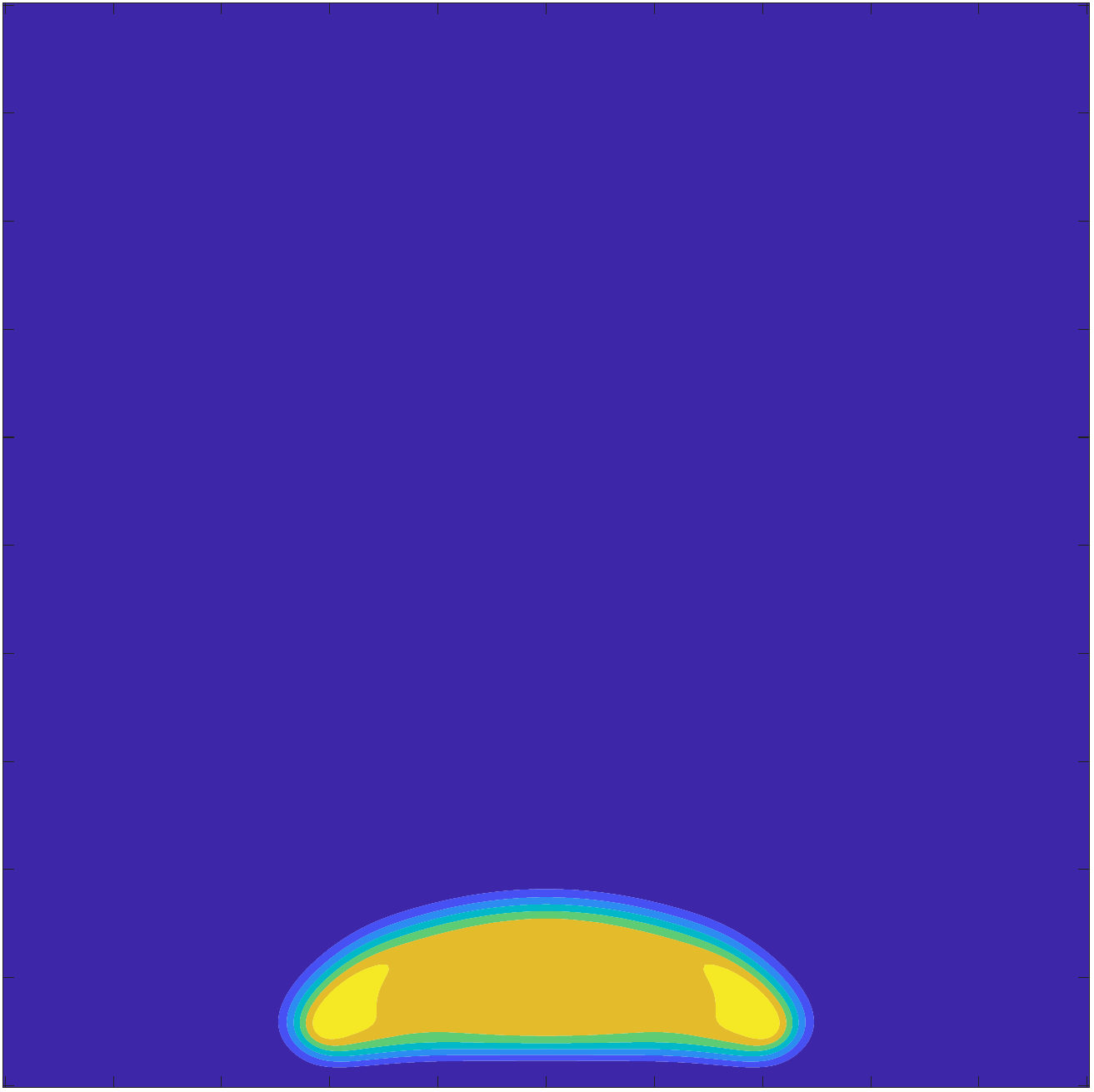}
	\includegraphics[scale=0.107]{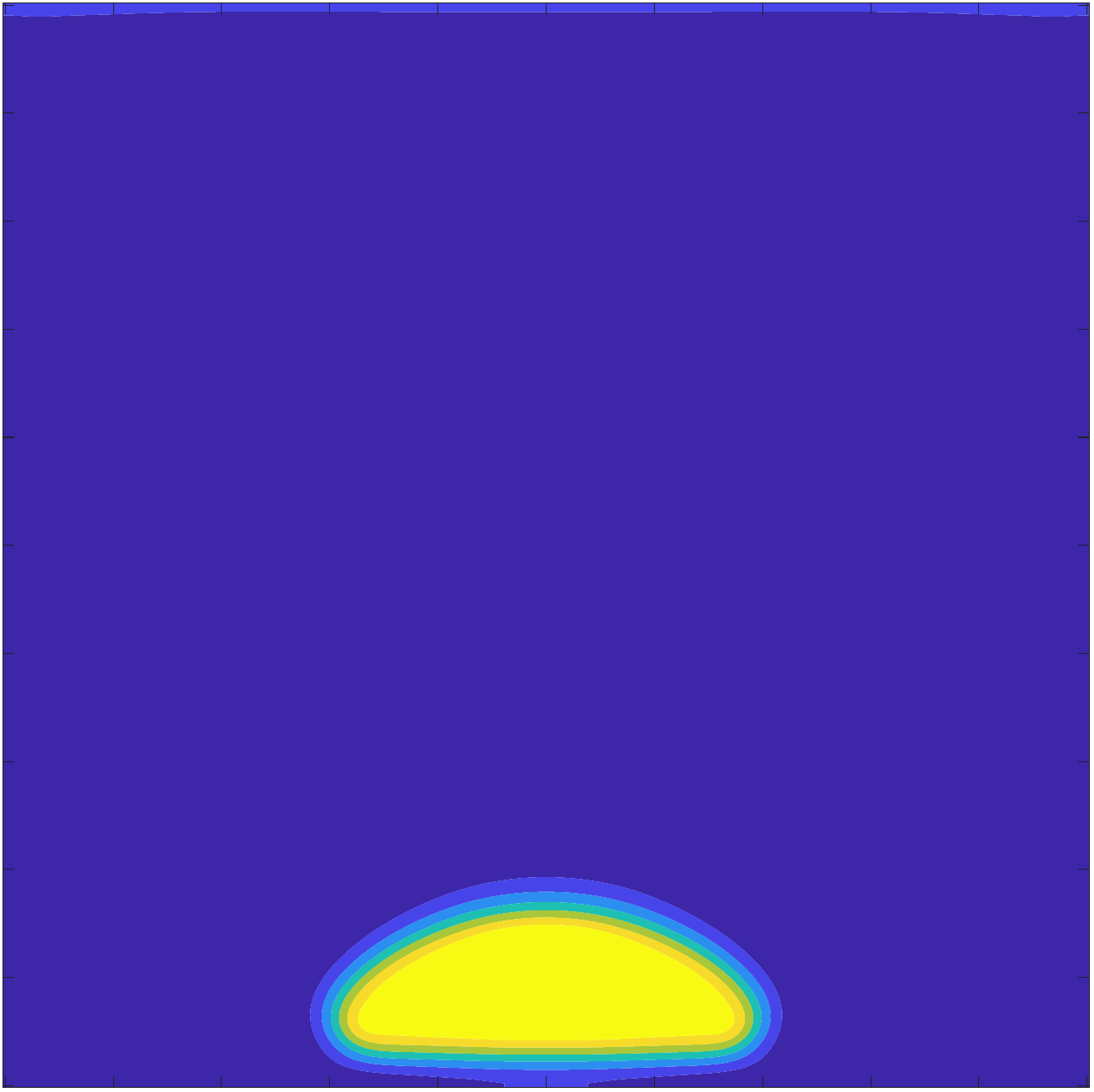}
	\includegraphics[scale=0.107]{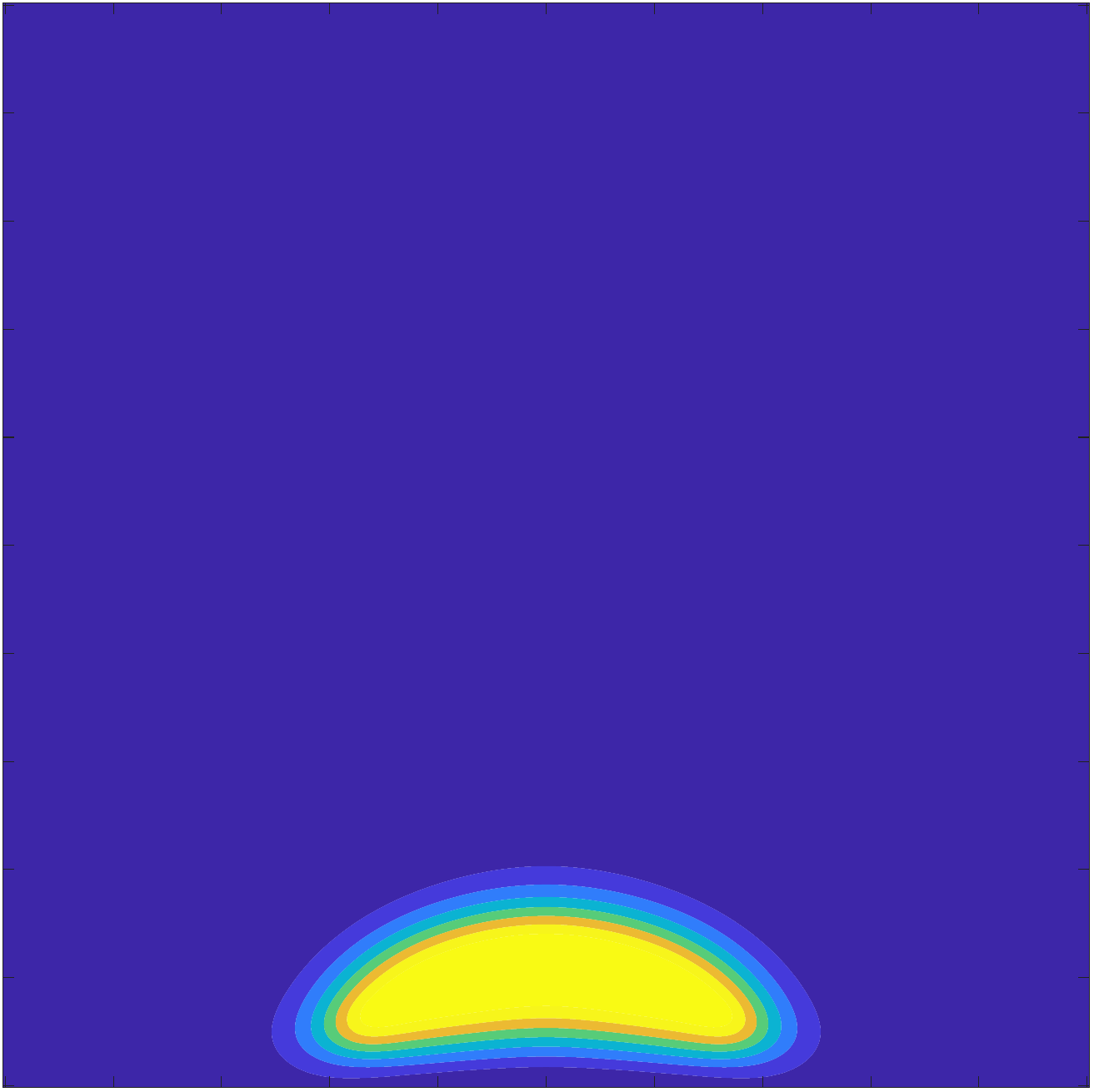}
	\includegraphics[scale=0.21]{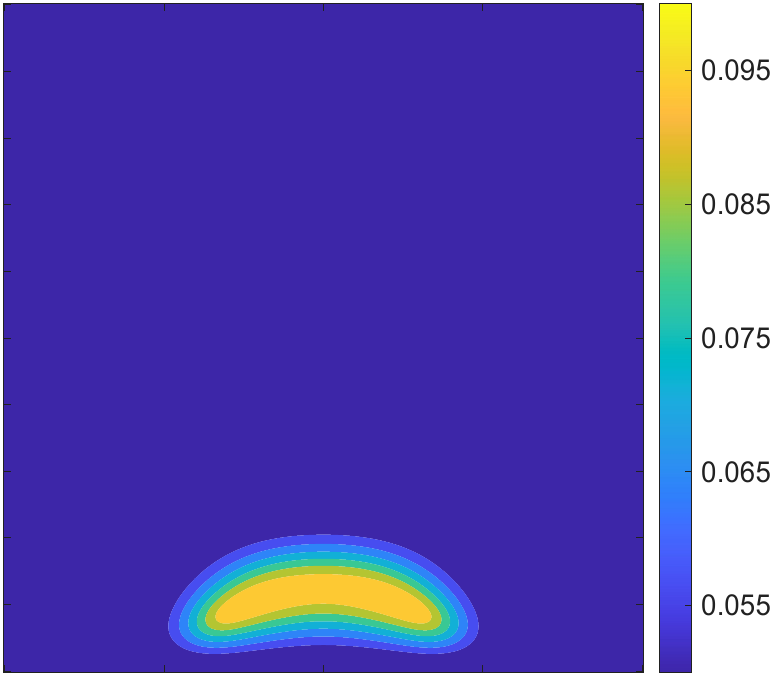}
	
	\vspace{1mm}
	\hspace{1mm}\includegraphics[scale=0.107]{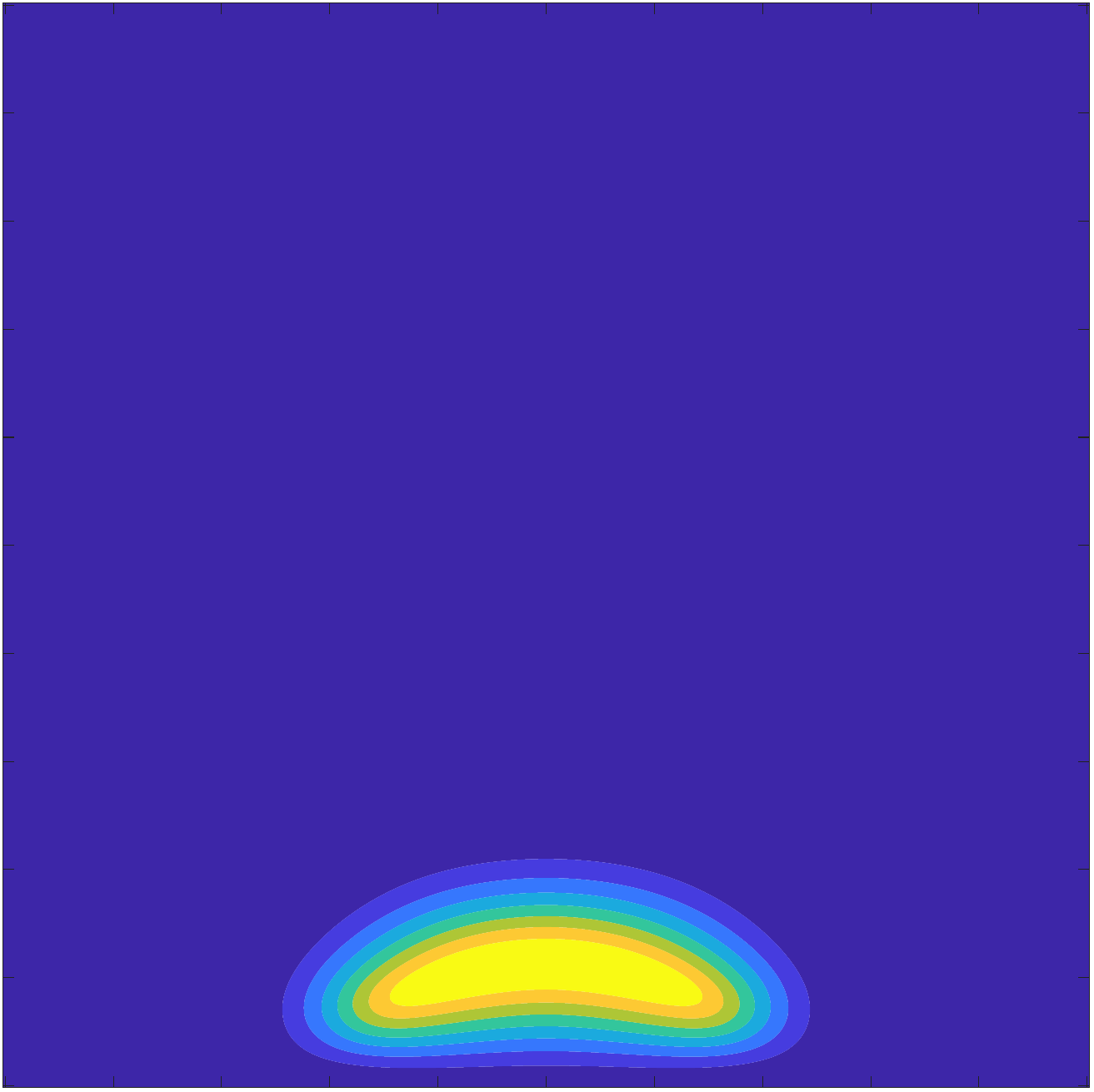}
	\includegraphics[scale=0.107]{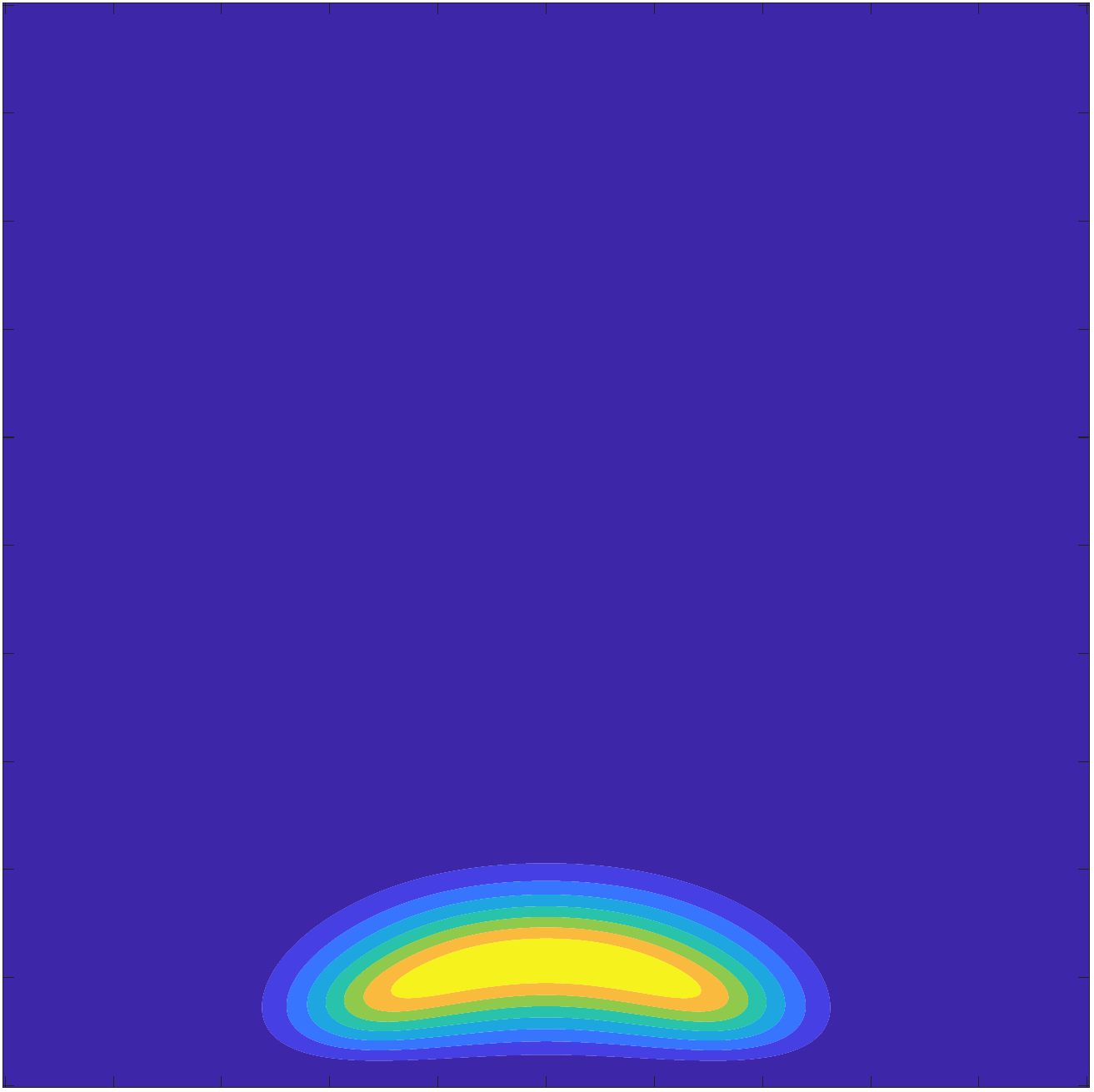}
	\includegraphics[scale=0.107]{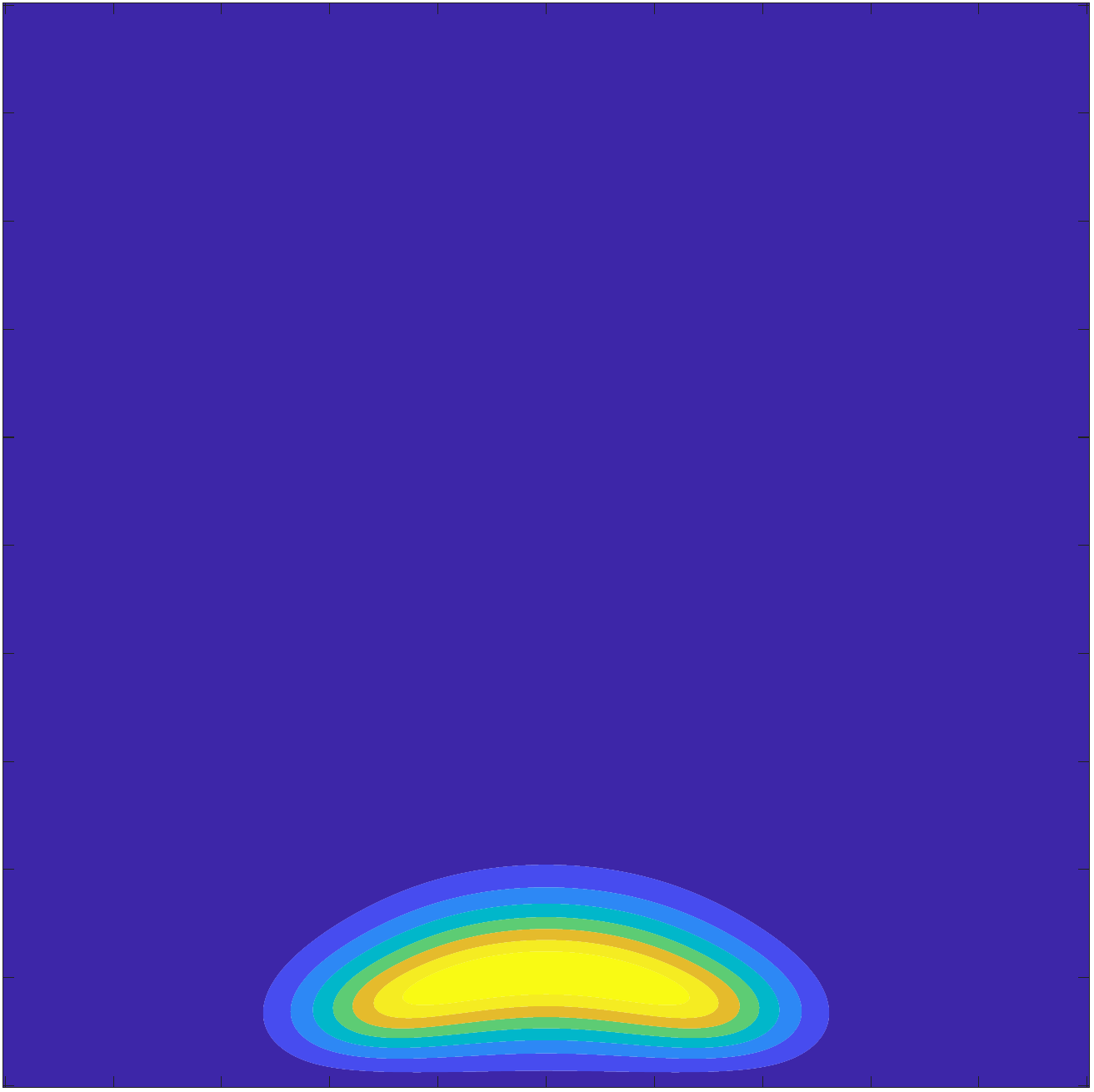}
	\includegraphics[scale=0.107]{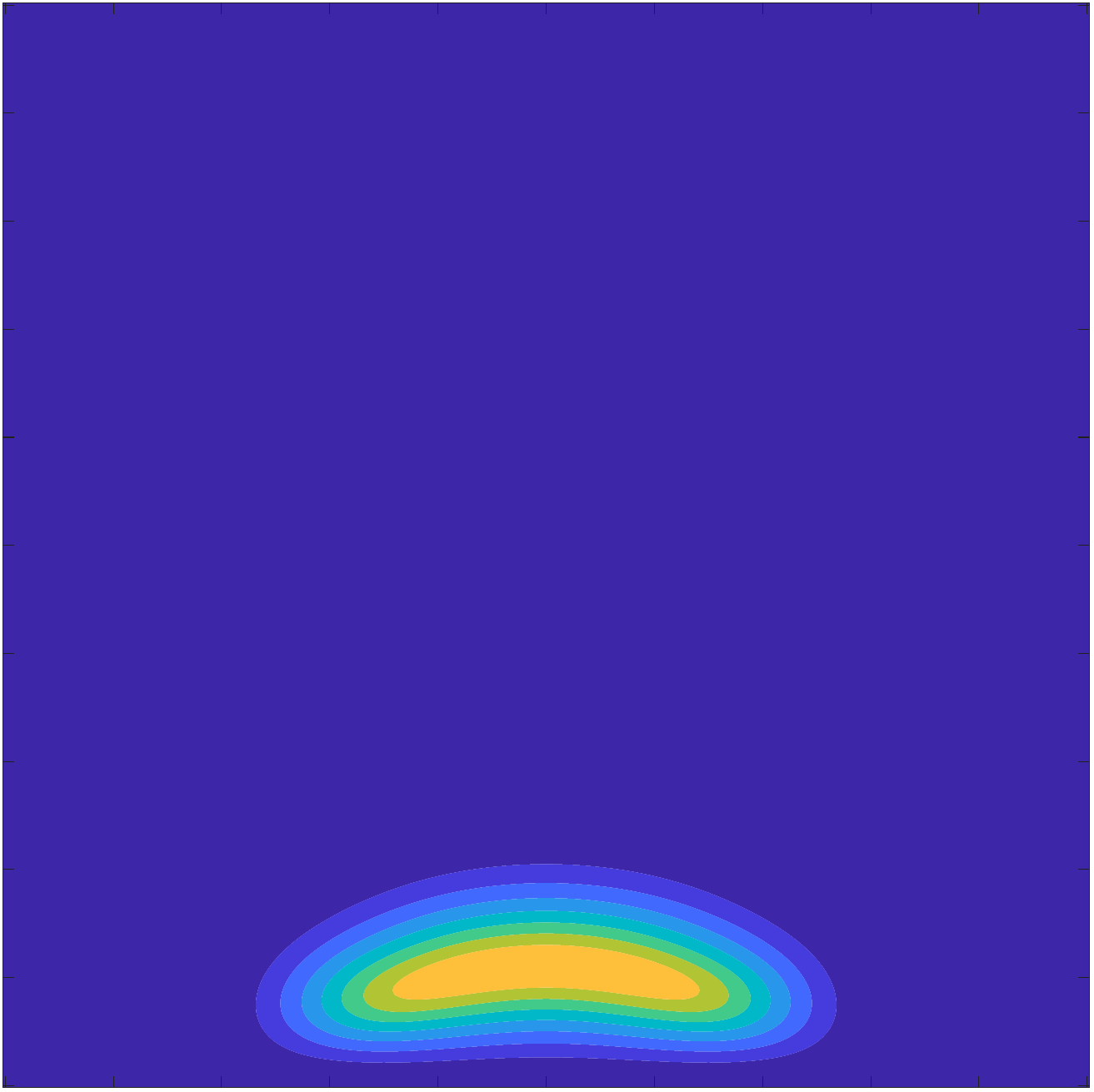}
	\includegraphics[scale=0.21]{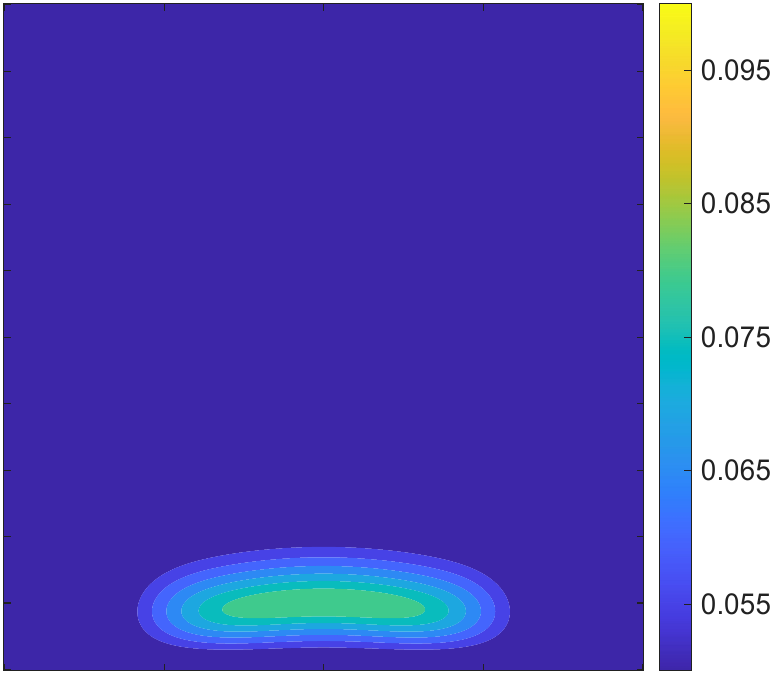}
	\caption{The evolution of the number density in the 2D Leidenfrost effect with semi-elliptical initial density. $t=0,20,40,80,120,140,160,180,200,300$.}
	\label{fig:Leidenfrost n 2}
\end{figure}
\begin{figure}[t]
	\centering
	\includegraphics[scale=0.107]{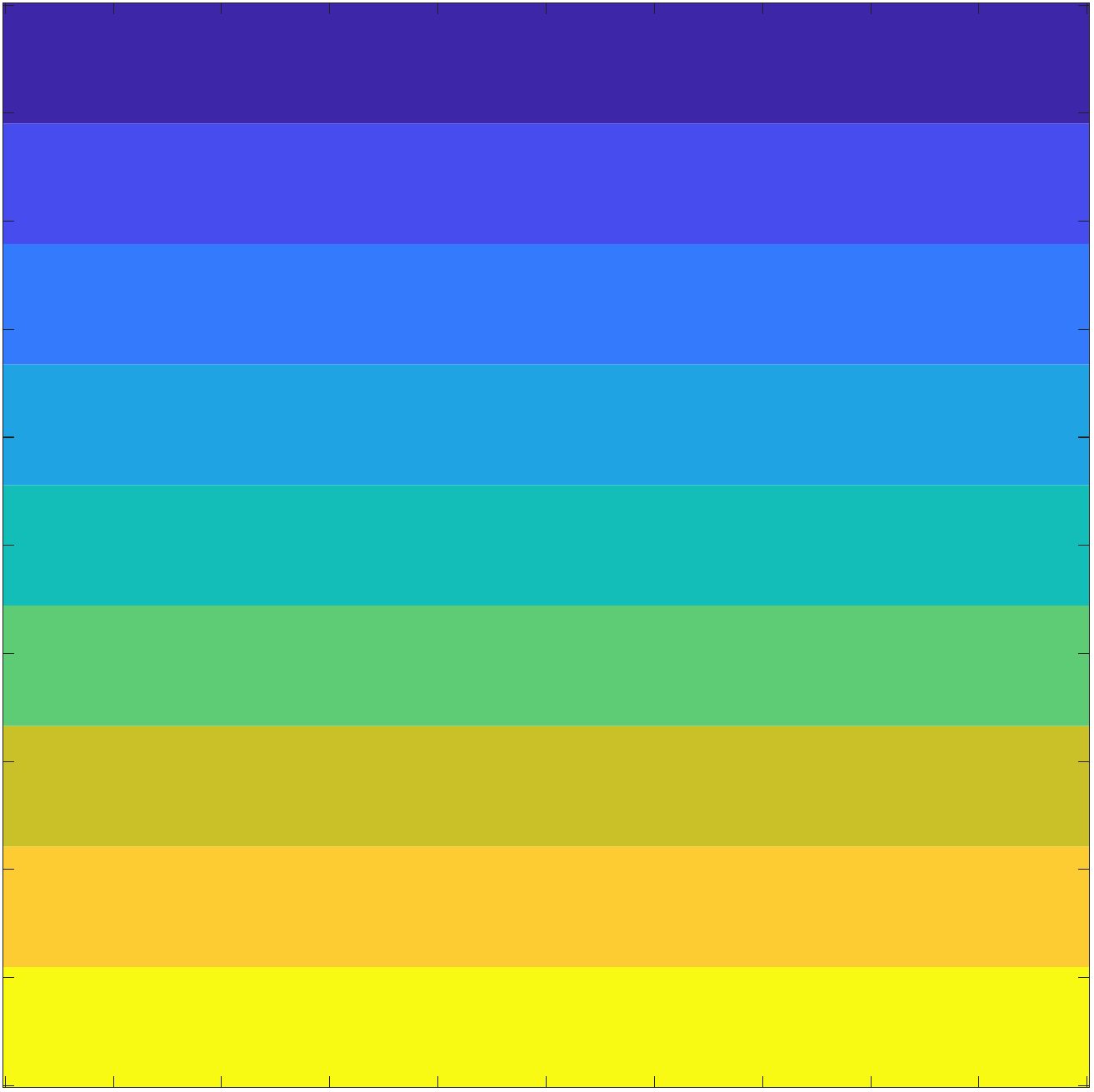}
	\includegraphics[scale=0.107]{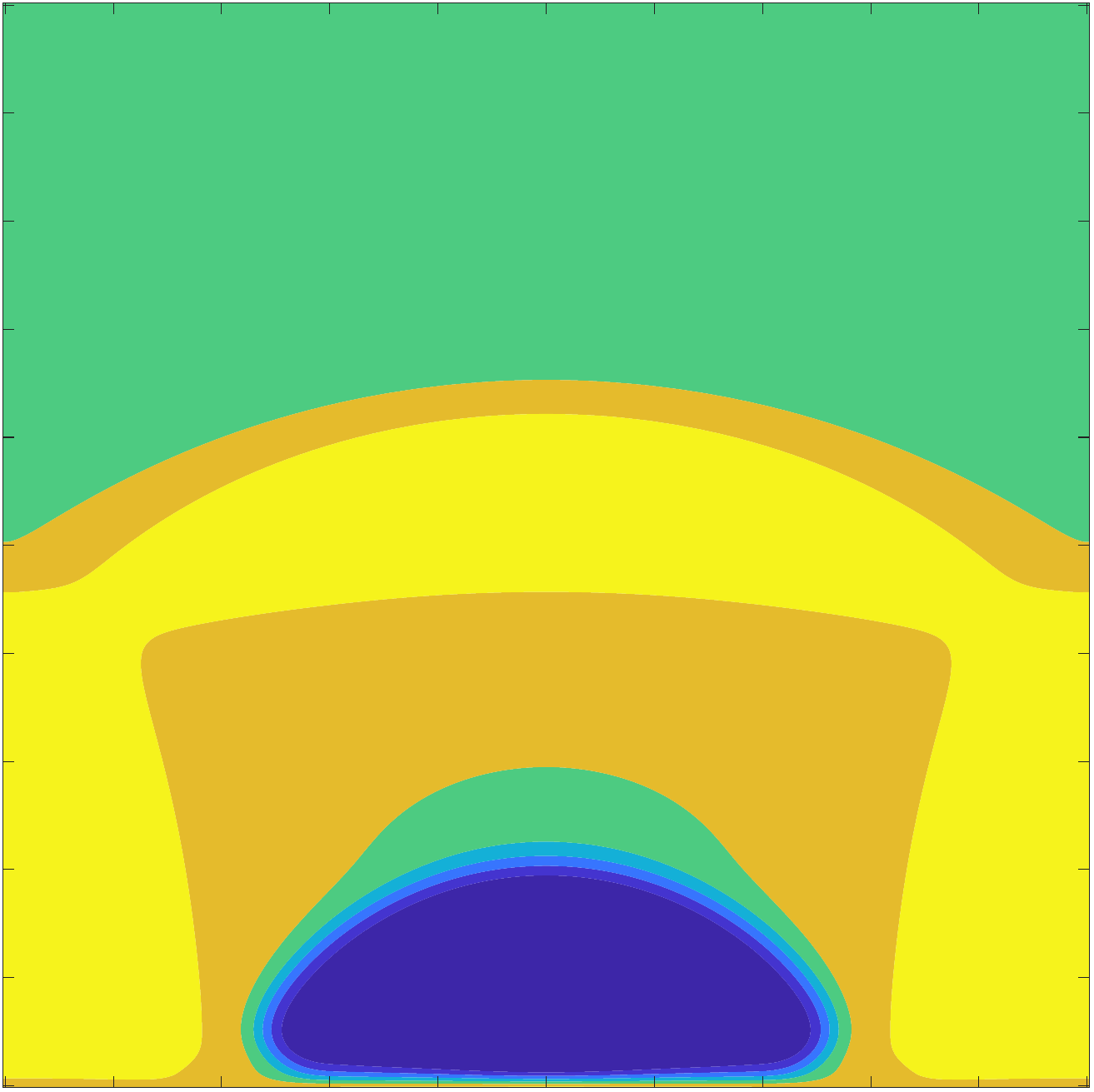}
	\includegraphics[scale=0.107]{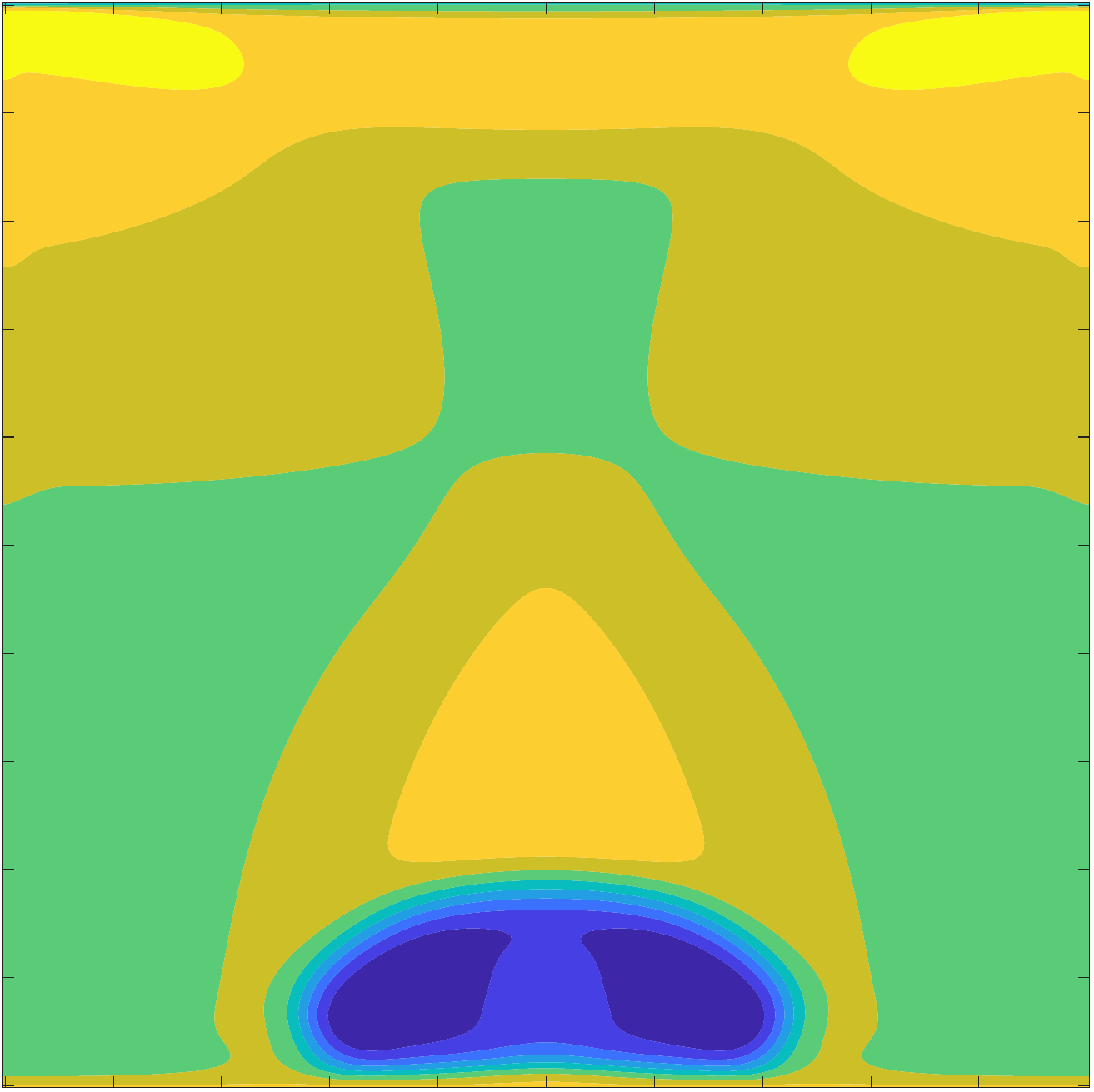}
	\includegraphics[scale=0.107]{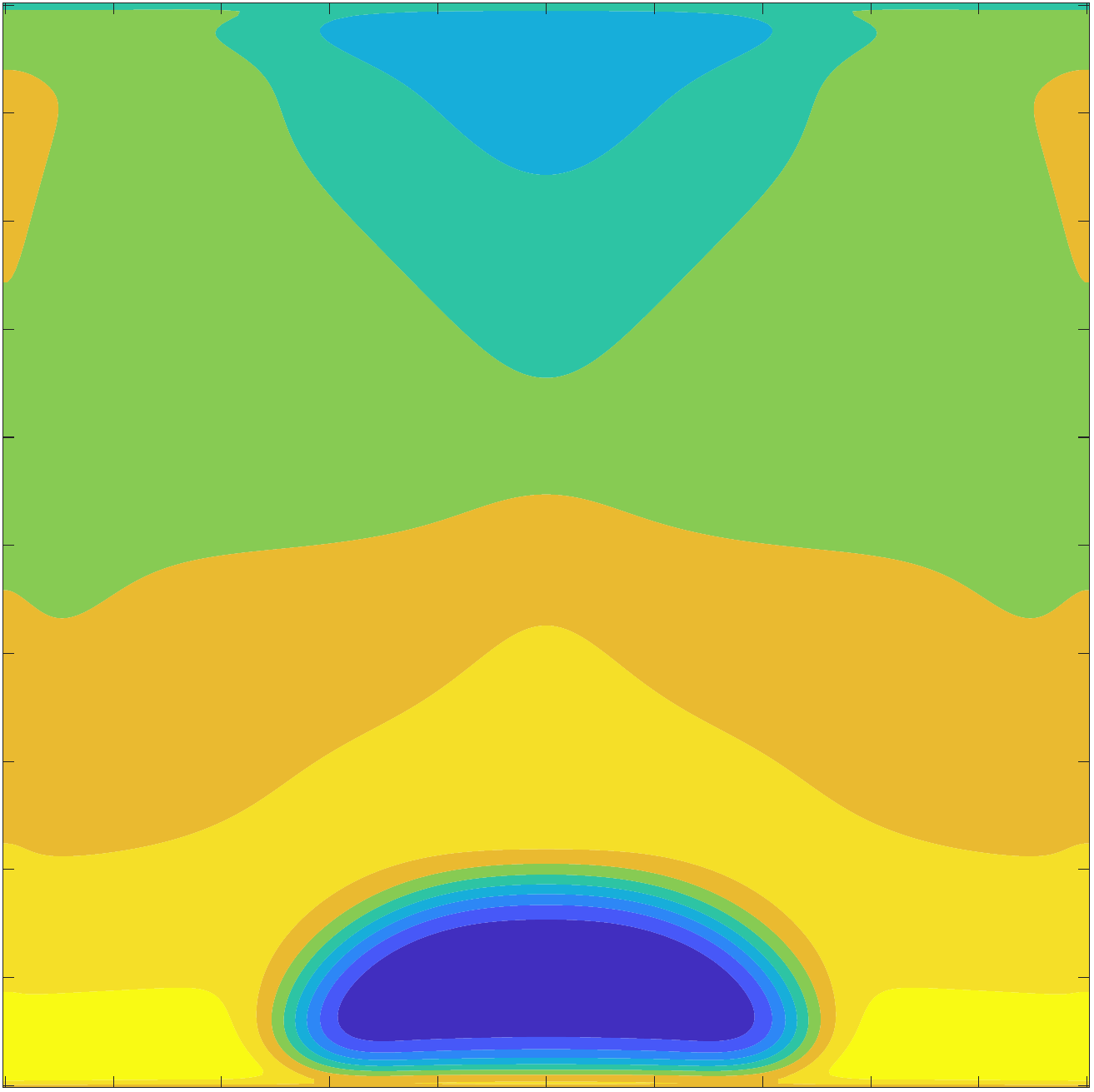}
	\includegraphics[scale=0.22]{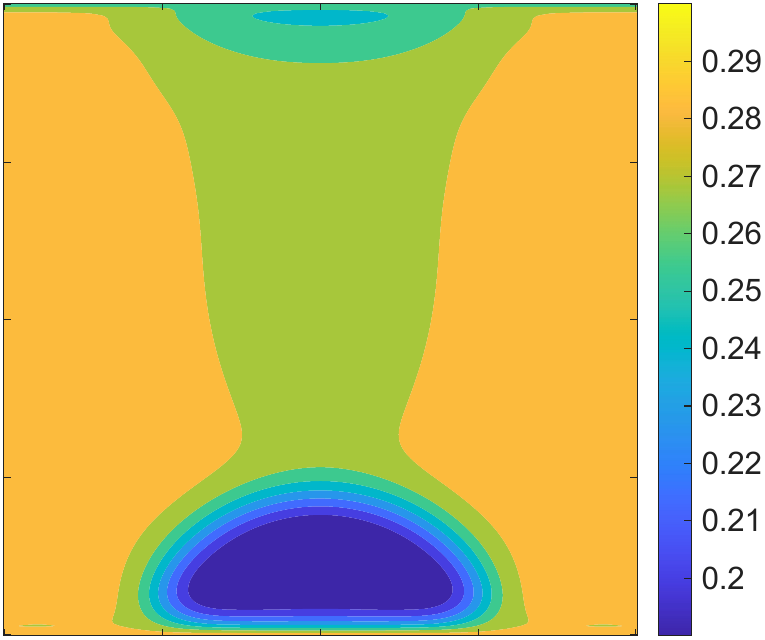}

    \vspace{1mm}
	\hspace{1mm}\includegraphics[scale=0.107]{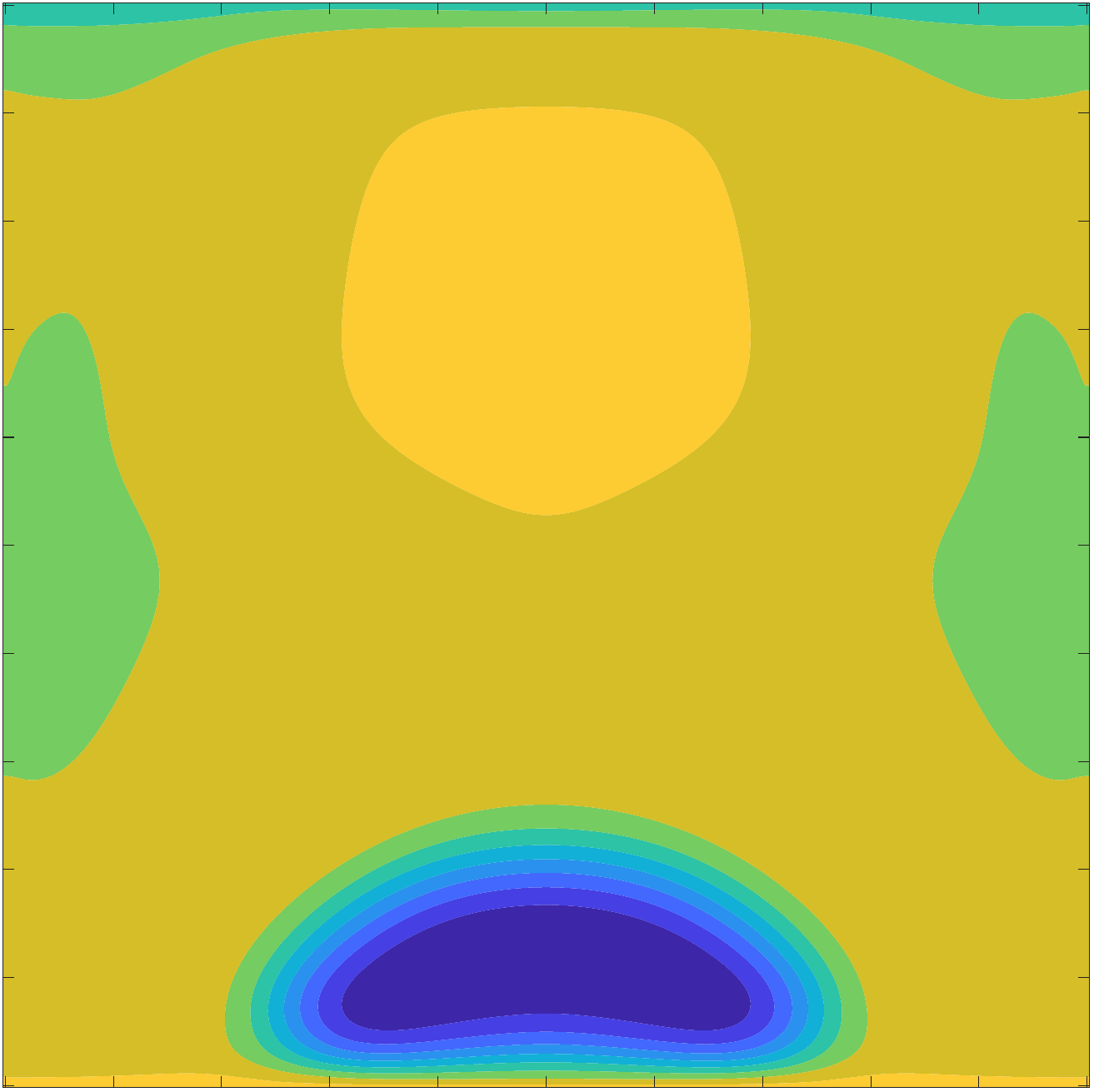}
	\includegraphics[scale=0.107]{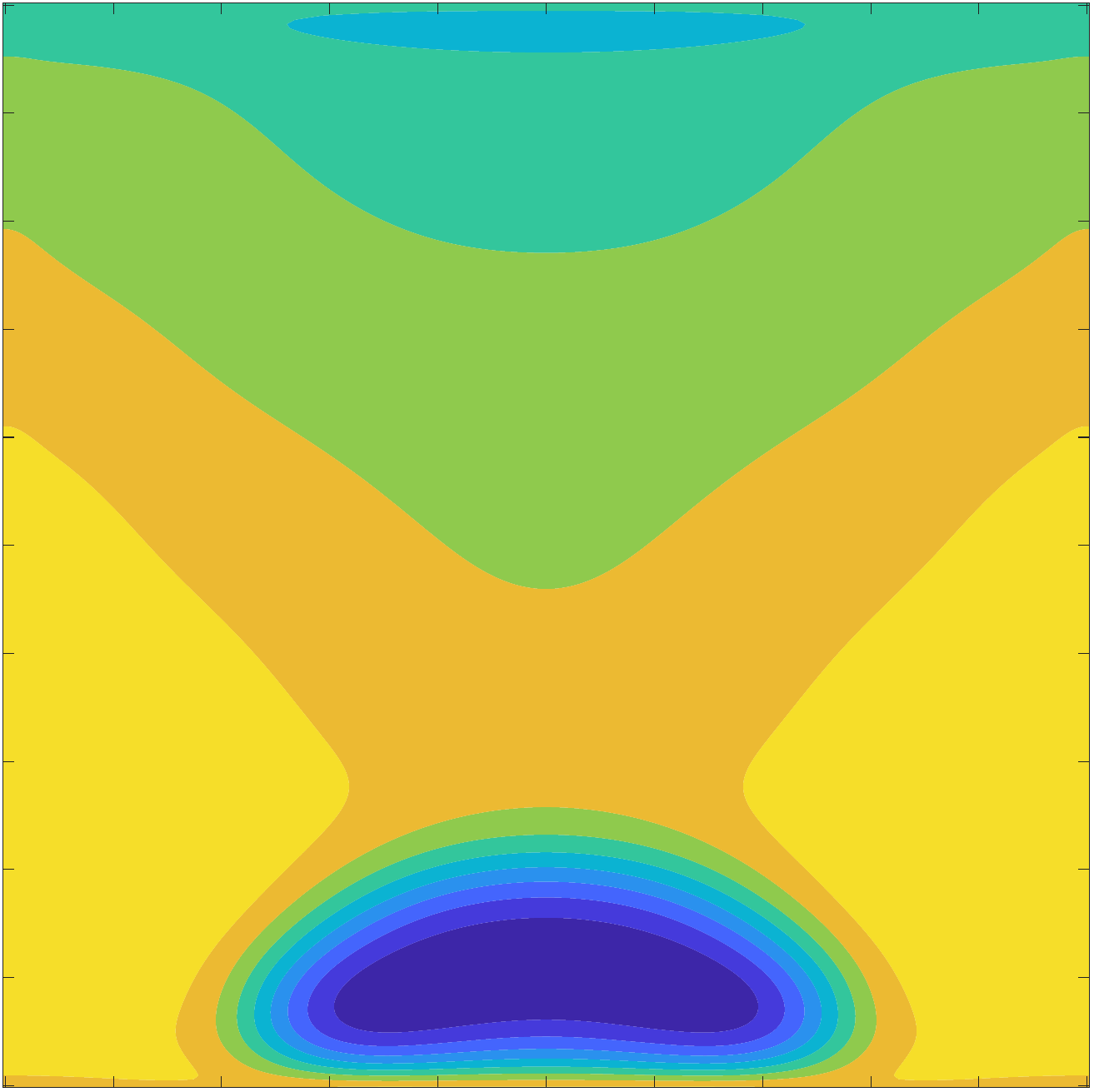}
	\includegraphics[scale=0.107]{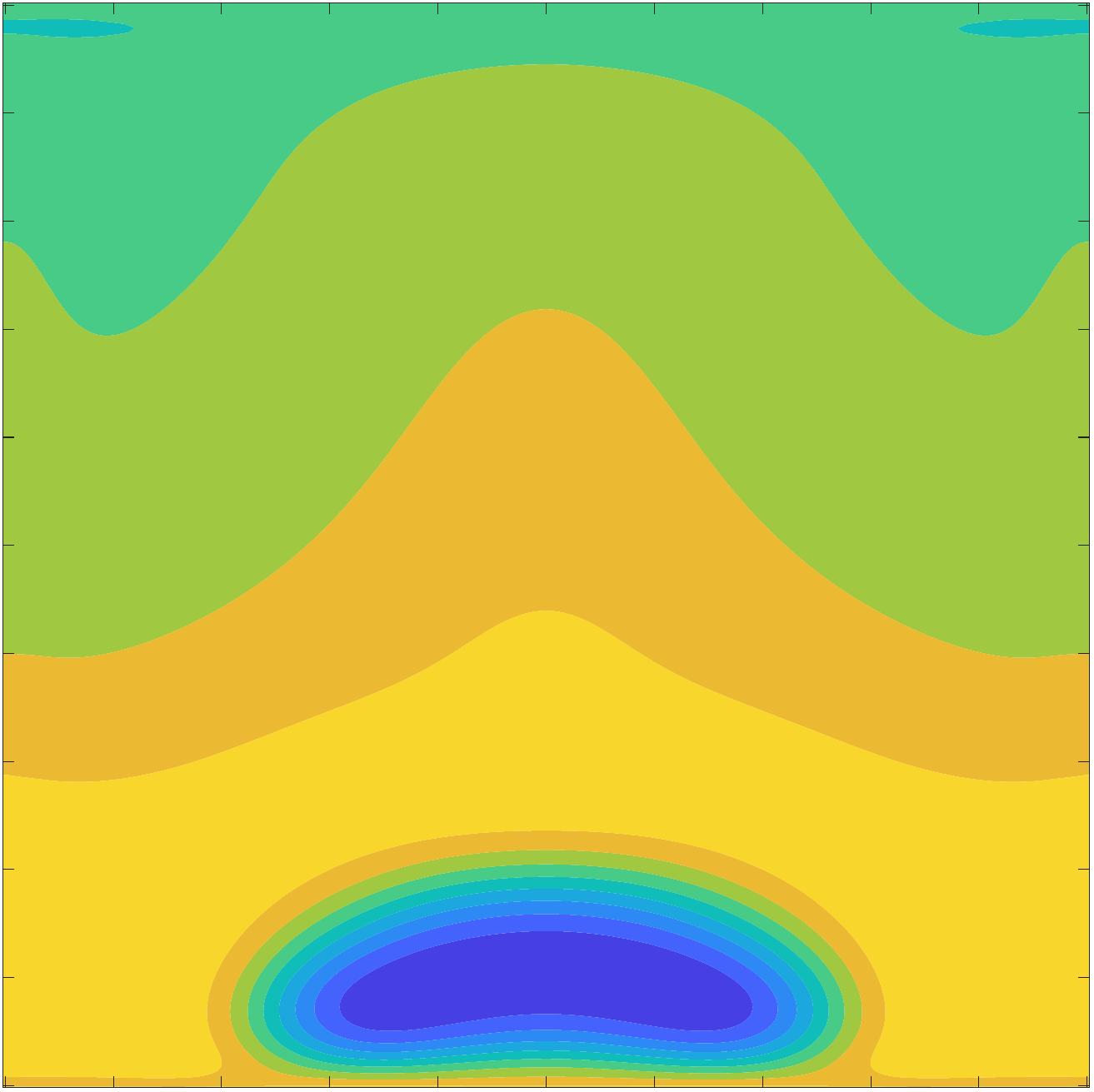}
	\includegraphics[scale=0.107]{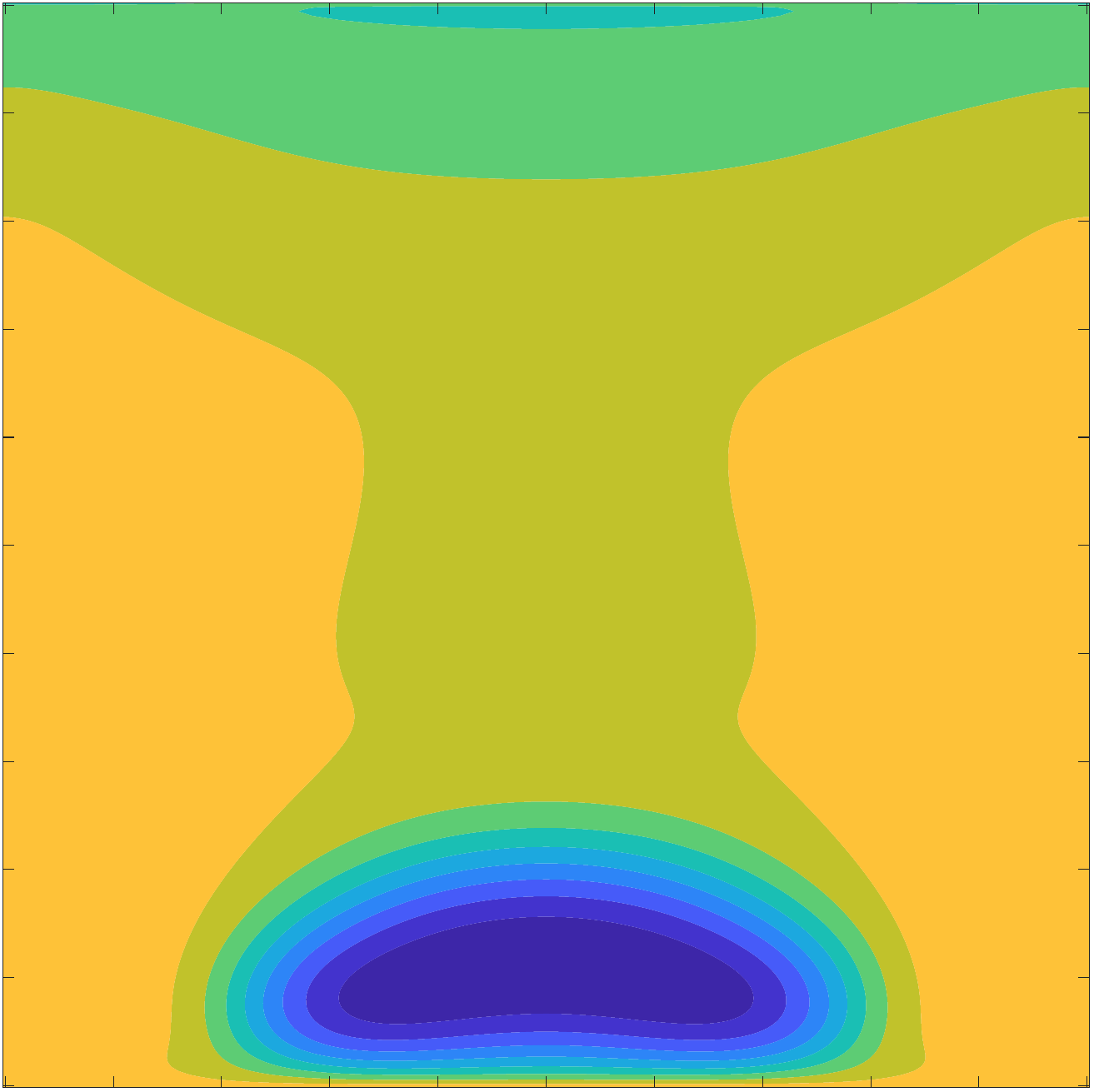}
	\includegraphics[scale=0.22]{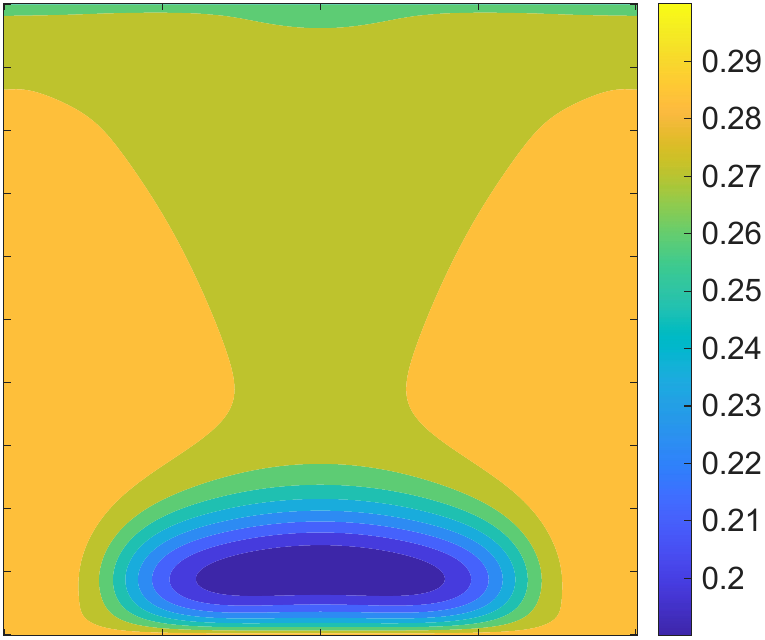}
	\caption{The evolution of the temperature in the 2D Leidenfrost effect with semi-elliptical initial density. $t=0,20,40,80,120,140,160,180,200,300$.}
	\label{fig:Leidenfrost T 2}
\end{figure}
\begin{figure}[t]
	\centering
	\includegraphics[scale=0.35]{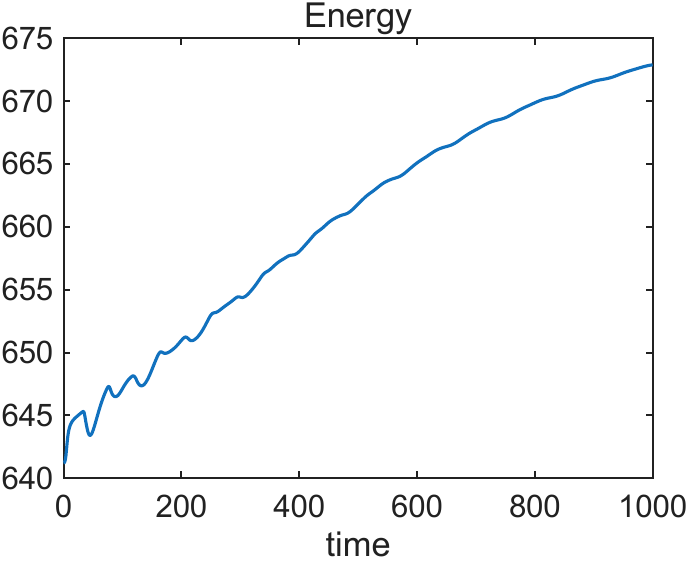}
	\includegraphics[scale=0.35]{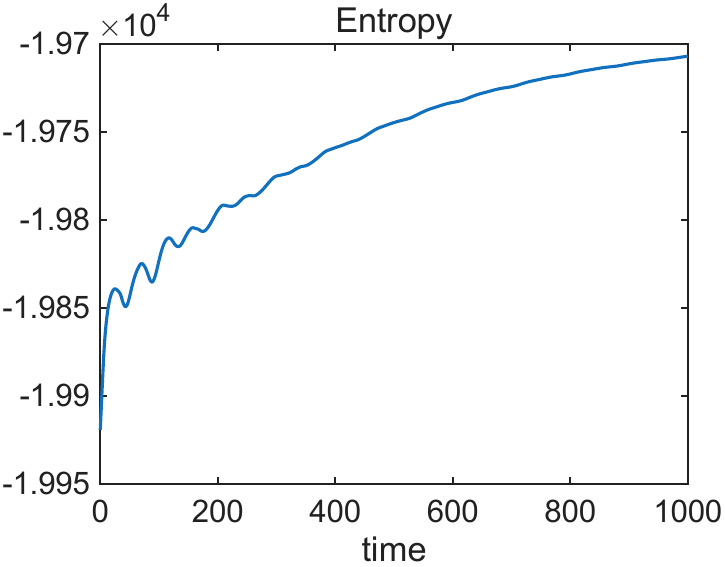}
	\includegraphics[scale=0.35]{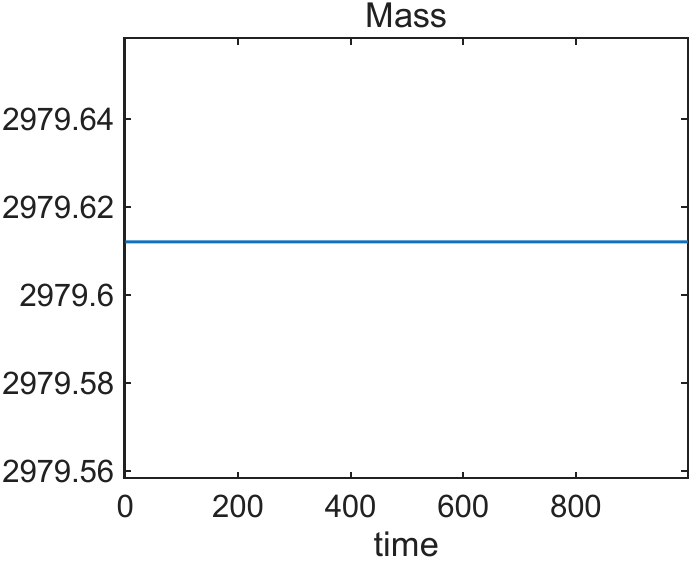}
	\caption{The variations of energy, entropy, and mass in the 2D Leidenfrost effect with semi-elliptical initial density.}
	\label{fig:Leidenfrost entropy 2}
\end{figure}

Owing to the high computational efficiency, the linear scheme allows us to perform three-dimensional simulations. The initial condition is a semi-ellipsoidal droplet in contact with a boundary kept at a high temperature, with the same liquid and vapor densities and an interface with a certain thickness. The temperature has higher Dirichlet boundary condition at $z=0$ and lower Dirichlet boundary condition at $z=L$ and varies linearly within the region. Taking computational complexity into account, the spatial step is reduced to $h=L/128$, keeping all other parameters identical to those in the two‑dimensional case.

Figure \ref{fig:Leidenfrost 3D} presents the three-dimensional evolution of the Leidenfrost droplet at selected times. The isosurface corresponding to the intermediate density $n_{mid}=0.1170$ is used to identify the liquid-vapor interface. For clarity, only half of the isosurface is displayed, together with density contours on the central cross-section. Similar to the two‑dimensional case, the droplet is lifted off the hot wall by the vapor, develops an upward indentation, and gradually evaporates.
Figure \ref{fig:Leidenfrost entropy 3D} shows the evolution curves of entropy and energy, and verifies the mass conservation property. The entropy and energy display an oscillatory increasing behavior comparable to the two‑dimensional case, yet the overall evolution is smoother. 

\begin{figure}[t]
	\centering
	\includegraphics[width=\linewidth]{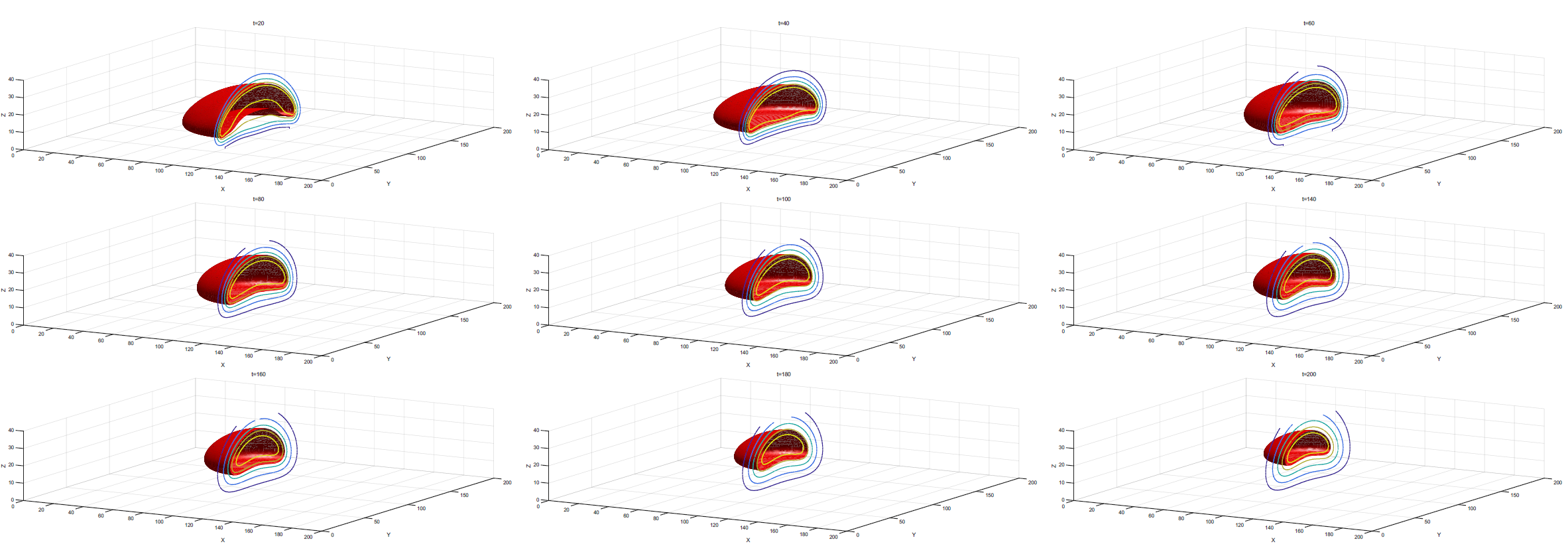}
	\caption{The evolution of the 3D Leidenfrost effect. The half isosurface corresponding to the intermediate density $n_{mid}=0.1170$ is used to identify the liquid-vapor interface, together with density contours on the central cross-section. $t=20,40,60,80,100,140,160,180,200$.}
	\label{fig:Leidenfrost 3D}
\end{figure}
\begin{figure}[t]
	\centering
	\includegraphics[scale=0.29]{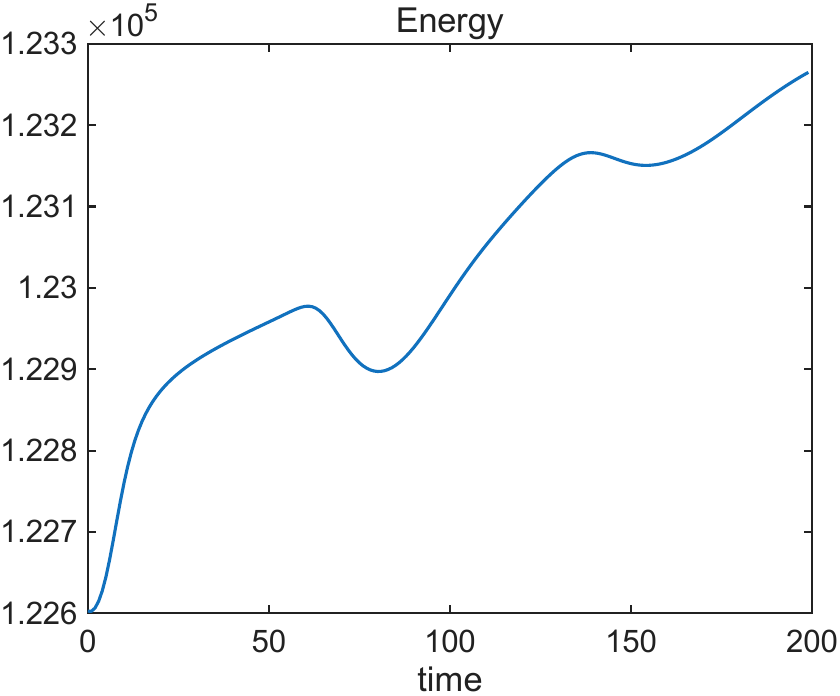}
	\includegraphics[scale=0.29]{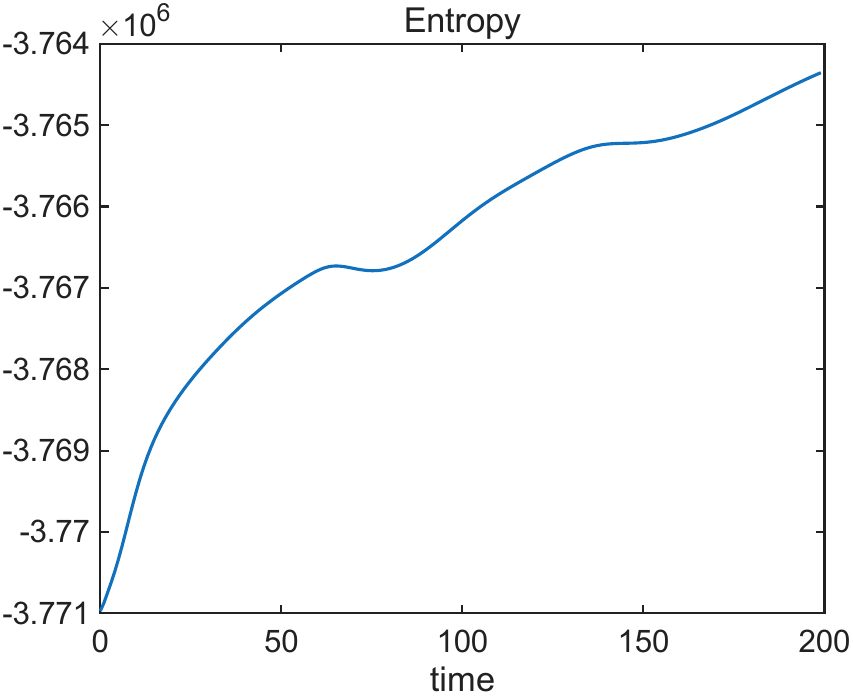}
	\includegraphics[scale=0.29]{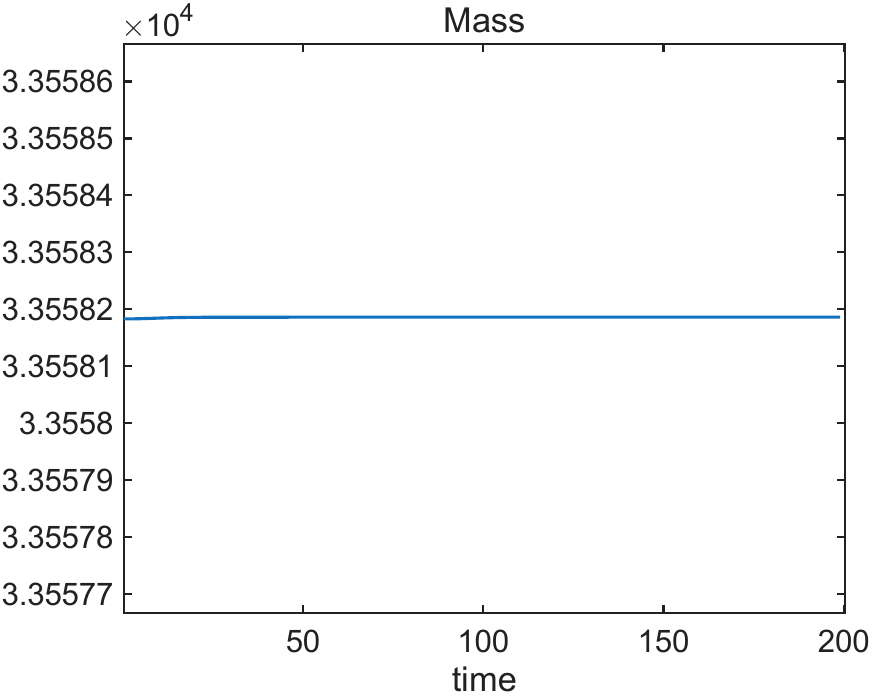}
	\caption{The variations of energy, entropy, and mass in the 3D Leidenfrost effect with semi-elliptical initial density.}
	\label{fig:Leidenfrost entropy 3D}
\end{figure}

\section{Conclusion} \label{sec: conclusion}
In this work, we have developed a linear, fully discrete, and decoupled numerical scheme for the dynamic van der Waals model of non-isothermal single-component liquid-vapor flows. 
With a suitable non-dimensionalization introduced, the original system is recast into an equivalent form. The reversible stress tensor is reformulated through the divergence identities to avoid relying on continuous chain rule at the discrete level.
Based on this reformulation, a component-wise convex-splitting strategy is applied to define the discrete chemical potential and temperature. The explicit treatment of the velocity in the mass and energy equations leads to a fully decoupled solution procedure. As a result, the density, the internal energy-temperature pair, and the velocity can be solved successively through linear subproblems at each time step. 
In particular, the decoupled density update can be viewed as an explicit treatment of the logarithmic term. Once the density is known, the temperature also depends linearly on the internal energy. 
A major feature of the proposed method is that it preserves the entropy-production law at the fully discrete level, which is one of the fundamental thermodynamic structures of the model. Numerical experiments verify the expected first-order temporal accuracy and second-order spatial accuracy and further demonstrate that the proposed method can effectively capture droplet evaporation, two- and three-dimensional Leidenfrost effects. These results indicate that the scheme is robust, efficient, and physically consistent.
Future work will mainly focus on two directions. One is the generalization of the numerical scheme to complex models, such as thermodynamically consistent boundary conditions and multi-component flows. The other is the development of fully discrete schemes that simultaneously maintain entropy production and energy conservation, with additional attention to positivity preservation of density and temperature.


\bibliographystyle{siamplain}
\bibliography{draft_ref}
\end{document}